\documentclass[12pt]{amsart}

\usepackage{amsmath,amsthm,amssymb,mathrsfs,amsfonts,verbatim,color,leftidx}
\usepackage{enumerate}
\usepackage{mathabx}
\usepackage{etoolbox} 
\usepackage{bbm}
\usepackage[all,tips]{xy}
\usepackage{graphicx,ifpdf}
\ifpdf
\fi
\usepackage{geometry}
\newtheorem{thm}{Theorem}[section]
\newtheorem{lem}[thm]{Lemma}
\newtheorem{cor}[thm]{Corollary}

\theoremstyle{definition}
\newtheorem{de}[thm]{Definition}
\newtheorem{ex}[thm]{Example}

\theoremstyle{remark}
\newtheorem{rem}[thm]{Remark}

\numberwithin{equation}{section}

\makeatletter
\newcommand{\rmnum}[1]{\romannumeral #1}
\newcommand{\Rmnum}[1]{\expandafter\@slowromancap\romannumeral #1@}
\makeatother

\newcommand{\br}{\mathbf{r}}

\newcommand{\bt}{\mathbf{t}}

\newcommand{\R}{\mathbb{R}}
\newcommand{\N}{\mathbb{N}}
\newcommand{\SL}{\operatorname{SL}}
\newcommand{\GL}{\operatorname{GL}}
\newcommand{\Z}{\mathbb{Z}}

\newcommand{\til}{\widetilde}

\newcommand{\dd}{\; \mathrm{d}}
\newcommand{\df}{{\, \stackrel{\mathrm{def}}{=}\, }}

\newcommand{\diag}{\mathrm{diag}}

\newcommand{\Ad}{\operatorname{Ad}}

\newcommand{\lip}{\operatorname{Lip}}
\newcommand{\supp}{\mathrm{supp\,}}

\begin{document}

\title{Central Limit theorem and cohomological equation on homogeneous spaces}


\author{Ronggang Shi}
\address{Shanghai Center for Mathematical Sciences, Jiangwan Campus, Fudan University, No.2005 Songhu Road, Shanghai, 200433, China}
 \email{ronggang@fudan.edu.cn}
\thanks{The author is supported by NSFC 11871158.}


\subjclass[2010]{Primary   37C85; Secondary 60F05, 37D30.}

\date{}


\keywords{homogeneous dynamics,  central limit theorem, cohomological equation}

\begin{abstract}
The dynamics of one parameter diagonal group actions on  finite  volume homogeneous spaces  has a partially hyperbolic feature. 
In this paper we extend the Liv\v{s}ic
type result to these possibly noncompact and nonaccessible systems. We also prove 
a central limit theorem for the Birkhoff averages of  points on a horospherical orbit. The Liv\v{s}ic type result allows us to show that the variance of the central limit theorem is nonzero provided that the test function has nonzero mean
 with respect to  an invariant probability measure. 
\end{abstract}

\maketitle

\markright{}

\section{Introduction}\label{sec;intro}

Let  $X=G/\Gamma$ be a finite volume
homogeneous space, where $G$ is a connected noncompact semisimple Lie group with finite center 
and $\Gamma$ is  an irreducible lattice of $G$. Recall that a lattice  $\Gamma$ is said 
to be irreducible if  for any noncompact simple factor $N$ of $G$ the   group $N\Gamma$ is dense in $G$. 
 The left translation  action of 
    $G$ on $X$ is (strongly) mixing
    with respect to the probability Haar measure $\mu$.
   The mixing is exponential  when the action of $G$ on $X$ has a strong spectral gap, i.e.,
    the action of   each noncompact simple factor of   $G$ on $X$ has a spectral gap. 
  Recall that the action of a closed  subgroup $H$ of $G$ on $X$ is said to have a spectral gap, 
   if there exists $\delta>0$ and a compactly supported probability measure $\nu$ on $H$ such that  
       for all 
    $ \varphi \in L^2_{\mu, 0}:=\{\varphi \in L^2_\mu: \mu(\varphi)=0  \}$
    one has 
  \[
  \int_X \left |\int_H \varphi(h^{-1}x) \dd \nu(h) \right |^2\dd\mu(x) \le (1-\delta)  \int_X |\varphi(x)|^2\dd\mu(x). 
  \]

      It was proved by Mozes \cite{mozes} that the  system $( X, \mu, G)$  is  mixing of all orders. 
    The effective version of multiple mixing was proved by
     Bj\"orklund-Einsiedler-Gorodnik \cite{beg} under the assumption of the existence of the  strong spectral gap. As an application, they obtained a central limit theorem in \cite{bg}. The aim of this paper is to give a 
     sufficient condition of nonzero variance and prove new central limit theorems 
     based on \cite{shi} where 
      the author proved effective 
   multiple correlations for the trajectory of certain measure $\nu$ singular to $\mu$. 
 
 The irreducibility  assumption for  $\Gamma$  is unnecessary, so we assume from now on that 
 $\Gamma$ is just a lattice of $G$. We consider
 the action of a one parameter $\Ad$-diagonalizable  subgroup   $F=\{a_t: t\in \R  \}$
  on $X$. Here $\Ad$-diagonalizable means the image of  $F$  through  the adjoint representation of $G$ is diagonal with respect to some basis of the Lie algebra.

Let $\phi : X\to \R $ be a continuous   function 
 integrable with respect to $\mu$,
and let $\nu$ be a probability measure on $X$. 
It is said that   $(\phi, \nu, F)$    obeys the  central limit theorem  if the random variables
\begin{align}\label{eq;random}
\frac{1}{\sqrt T}\int_0^T \phi(a_t  x)- \mu(\phi)\dd t \quad \mbox{ given by } (X, \nu )
\end{align}
converge as $T\to \infty$ to the  normal distribution with mean zero and variance 
$\sigma=\sigma(\phi, \nu, F) \ge 0$,
 i.e., for any bounded  continuous   function $f:\R\to \R$, 
 \begin{align}\label{eq;more}
\lim_{T\to \infty}\int_X f\left (\frac{1}{\sqrt T}\int_0^T \phi(a_t  x)- \mu(\phi)\dd t\right ) \dd\nu(x)=
\frac{1}{\sqrt{2\pi \sigma}}\int_{-\infty}^\infty f(s) e^{-\frac{s^2}{2\sigma}} \dd s.
 \end{align}
If $\sigma=0$, then the right hand side of (\ref{eq;more}) is interpreted to be $f(0)$ and we say the central limit theorem is degenerate. 

Let $C_c^\infty(X)$ be the space of compactly supported  smooth and real valued functions on $X$.
It was proved in \cite{bg} that  $(\phi, \mu, F)$  obeys the central limit theorem  
if $\phi\in C_c^\infty (X)$ and the action of $F$ on $X$ is exponential mixing of all orders, e.g.,  the action of $G$ on $X$ has a strong spectral gap \cite{beg}. Some special cases were known earlier in  \cite{sinai}\cite{ratner}\cite{lejan} when 
 $G $  is the isometric group of a hyperbolic manifold with constant negative curvature.  
If in addition  $X$ is compact, then
 the characterization of  nonzero variance  is well understood. 
It was proved by Ratner \cite{ratner} that $\sigma(\phi, \mu, F)=0$   if and only if  
the system of  cohomological equations parameterized by $s> 0$
\begin{align}\label{eq;homology}
\int_0^s \phi(a_t x)-\mu(\phi)\dd t=\varphi(a_sx)-\varphi( x) 
\end{align}
has a measurable  solution $\varphi\in L^2_\mu$. 
Here a measurable solution means (\ref{eq;homology}) holds for $\mu$ almost every $x\in X$. 
Using Liv\v{s}ic's theorem \cite{livsic} which says that a measurable solution to (\ref{eq;homology}) is equal to a continuous function almost everywhere,   Melbourne and T\"or\"ok  \cite{mt} showed that $\sigma(\phi, \mu, F)=0$ if and only if 
\begin{align}\label{eq;periodic}
\int_0^s \phi(a_t x)-\mu(\phi)\dd x=0\quad \mbox{for all } x\in X \mbox{ and } s>0 \mbox{ with } a_sx=x. 
\end{align}

 Our first main result is to extend Liv\v sic's theorem to the  homogeneous space $X$. 
 This will allow us to give a sufficient condition of nonzero variance similar to (\ref{eq;periodic}).
For $a\in G$, we use  $G_a'$ to denote  
the group 
generated by  the unstable horospherical subgroup $G_a^+$  and the stable horospherical subgroup $G_a^-$, where
\begin{align*}
G^+_{a}=\{g\in G:  \lim_{n\to \infty}a^{-n } g a^n\to 1_ G \} \quad \mbox{and }\quad
G^-_{a}=\{g\in G:  \lim_{n\to \infty}a^{n } g a^{-n}\to 1_ G \}.
\end{align*}
Here and hereafter $1_G$ denotes the identity element of a group $G$. 
   The group $G_a'$ is
 a connected  semisimple Lie group without compact factors and it is 
  normal in  $G$.
Let $ \widehat C_c^\infty (X)=C_c^\infty (X)+\R $  be the set of functions which can be written as a sum of a  function in $C_c^\infty(X)$ and a constant.

       \begin{thm}\label{thm;continuous}
       	Let  $X=G/\Gamma$ where $G$ is a connected semisimple Lie group with finite center and  $\Gamma$ is a lattice.  Let $\{a_t : t\in 
       	\R \} $ be a one parameter $\Ad$-diagonalizable subgroup of 
       	$G$ and let  $\mu$ be  the probability Haar measure on $X$.  Suppose $a=a_1$ 
      and  the action of  $G_a'$  on $X$ has a spectral gap. 
       	If $\varphi$ is  a 
       	measurable solution  to the cohomological equation $\psi(x)=\varphi(ax)-\varphi(x)$ where $\psi\in \widehat C_c^\infty (X)$, then $\varphi\in L^2_\mu$  and there is a smooth
       	function
       	$\til \varphi: X\to \R$ such that $\varphi=\til \varphi$  almost everywhere with respect to $\mu$.
       \end{thm}

   In the case where $X$ is compact and accessible,
    Theorem \ref{thm;continuous} is a special case of
   Wilkinson 
   \cite[Thm.~A]{wilkinson}. 
   The accessible assumption in our setting is the same as $G_a'=G$.  
   Our result  is new in the case where $X$   nonaccessible or  $X$ is noncompact. 
   The spectral gap assumption is always satisfied 
   if $G_a'$ is nontrivial, $G$ has no compact factors and $\Gamma$ is irreducible, see Kelmer-Sarnak \cite{ks}. 
We prove the smoothness of $\varphi$ along $G_a'$ orbits using the method of 
 \cite{wilkinson} and  Avila-Santamaria-Viana \cite{asv}. The method allows us to prove  the
  H\"older continuity
 of $\varphi$ along $G_a'$ orbits if we only  assume $\psi$ is   H\"older. 
The H\"older continuity in the case where   $X=\SL_d(\R)/\SL_d(\Z)$  is due to Le Borgne \cite{le}. 
 This part doesn't use the spectral gap assumption. 
 To prove the smoothness   of $\varphi$ along  the central foliations  of $a$
  we 
 need to  use 
 the spectral gap assumption which implies the dynamical system $(X, \mu, a)$ is exponential mixing. 
Here we use  some ideas from 
Gorodnik-Spatzier \cite{gs} and 
Fisher-Kalinin-Spatzier \cite{fks}.
  The smoothness of $\varphi$ then follows from a result of Journ\'e \cite{journe} which says that   a function uniformly   smooth on 
  transverse family of foliations is  smooth. 
 We will prove 
 Theorem \ref{thm;continuous} as well as its refinements where $\psi$ is only assumed to be H\"older continuous 
 in  \S \ref{sec;leaf} and \S \ref{sec;continuous}.

 Now we state our result on the central limit theorem   and give sufficient conditions for  the  
 nonzero variance. In order to define 
 the measure $\nu$ singular to $\mu$ we need to 
 review some concepts from \cite{shi}. 
We first setup the notation. 
 \begin{enumerate}[\quad (n.1)]
 	\item $H$ is a connected  normal subgroup of $G$ without compact factors.
 	\item  $P$  is an absolutely  proper 
 	parabolic subgroup of $H$, i.e.,   $P$  contains none of the simple factors of $H$.
 	\item $U$ is  the unipotent radical of $P$.
 	\item  $A$ is a maximal $\Ad$-diagonalizable subgroup 
 	of $H$ contained in  $P$. 
 \end{enumerate}

   The group  $U$  is said to be  $a$-expanding for some 
      $a\in A$ 
      if, for any nontrivial irreducible   representation 
      $\rho: H \to \GL(V)$ on a finite dimensional real vector space $V$, one has 
      \[
      \lim_{n\to \infty} \rho(a^{-n}) v=0 \quad \forall\ U\mbox{-fixed } v\in V. 
      \]
      The expanding cone of $U$ in $A$ is 
      \[
      A_U^+=\{a\in A: U \mbox{ is } a\mbox{-expanding}  \}
      \]
      and it has nice dynamical properties while translating $U$-slices on $X$.
      One of the main results of \cite{shi} is to explicitly describe the 
      expanding  cone
       $A_U^+$ of $U$
      in $A$.  We postpone the explicit description of $A_U^+$ to \S\ref{sec;review} where 
      we also review effective multiple  correlations. Here we  only give an example which is the motivation of this concept. 
            \begin{ex}\label{ex;main}
      	Let $m, n$ be positive integers.
      	\begin{align*}
      	H=G=\SL_{m+n}(\R), 
      	U =\left \{
      	\left(
      	\begin{array}{cc}
      	1_m & h\\
      	0_{nm} & 1_n 
      	\end{array}
      	\right)\in H: h\in \mathrm{M}_{mn}(\R)
      	\right\} 
      	\end{align*}
      	where $1_m$ and $1_n$ are the identity matrices of rank $m$ and $n$, respectively,
      	$A$ is the identity component of diagonal matrices in $G$,  and 
      	\begin{align}\label{eq;cone example}
      	A^+_U=\{\diag(e^{r_1 }, \ldots, e^{r_m}, e^{-t_1}, \ldots, e^{-t_n})\in G: r_i, t_j>0\}.
      	\end{align}
      \end{ex}

Suppose   $p\in A$  is a  regular element, where regular refers to    the projection of $p$ to each simple factor of $H$ is not the identity element. 
Let $H_p^+$ be 
the unstable horospherical subgroup of $p$ in $H$.
We  assume that $H_p^+\le U$,  $a_1\in A_U^+$ and 
the conjugation of $a_1$ expands $H_p^+$, i.e.,
\[ 
(\star) \qquad
\{a_t:t>0 \}\subset A_{U,p}^{+}\df\{a\in A_U^+ : H_p^+\le   H_a^+ \cap U \}.
\]
The assumption  $(\star)$ is always checkable due to the explicit description of $A_U^+$, see \S \ref{sec;review} for more details. 
 
Let  $\nu=\nu_{q, z}$ be a probability measure on $X$  given 
by  a  compactly supported nonnegative  smooth  function $q$  on 
 $U$  and $z\in X$  
in the following way:
 \begin{align}\label{eq;measure}
 \nu_{q, z}(\varphi)=\int_U \varphi(uz)q(u )\dd \mu_U(u) \quad \forall \varphi\in C_c(X),
 \end{align}
 where $\mu_U$ is a fixed  Haar measure on $U$  with   $\mu_U(f)=1$. 
 Now we are ready to state the central limit theorem.   
\begin{thm}
	\label{thm;central h}
	Let $X=G/\Gamma$ and $\mu$ be as in Theorem \ref{thm;continuous}. 
	Let $H, P, U$ and $A$ be as in $(n.1)$-$(n.4)$. 
Suppose  $F=\{a_t:t\in 
\R \}$ is  a one parameter subgroup of  $H$ satisfying $(\star)$ for a regular element $p\in A$
and 
$\nu=\nu_{q, z}$ is  a probability measure on $X$ given by  (\ref{eq;measure})
for $z\in X$ and $q\in C_c^\infty (U)$. 
Then for all 
$\phi\in \widehat C_c^\infty (X)$ with $\int_X \phi \dd\mu=0$,
 the system  $(\phi, \nu, F)$ obeys the central limit theorem  with variance
\begin{align}\label{eq;variance 0}
\sigma( \phi, F)=\lim_{l\to \infty}2\int_0^l \int_X \phi(a_t x)  \phi(x) \dd\mu(x)\dd t\in [0, \infty).
\end{align}
\end{thm}

It can be seen from (\ref{eq;variance 0}) that the variance $\sigma(\phi, F)$ does not depend 
on $\nu$ and is the same as   $\sigma(\phi, \mu, F)$. So the sufficient condition
for the nonzero $\sigma(\phi, F)$ below can also be applied to the central limit theorem proved in \cite{bg}.
In the setting of Example \ref{ex;main} with $\Gamma=\SL_{m+n}(\Z)$, 
  Theorem \ref{thm;central h} is due to Bj\"orklund-Gorodnik \cite{bgg}.
The difference between our proof and that in \cite{bgg} is that we do not use commulants.

      \begin{thm}\label{thm;variance}
      	Let the notation and the assumptions  be as in Theorem \ref{thm;central h}. 
Then the variance  $\sigma(F, \phi)$ in (\ref{eq;variance 0}) is equal to zero if and only if 
the system of  cohomological equations 
\begin{align}\label{eq;guoqing}
	\int_0^s \phi(a_t x)\dd t=\varphi(a_sx)-\varphi( x)  \qquad  (s> 0 ) .
\end{align}
 has  a measurable solution $\varphi\in L^2_\mu$.
    \end{thm}
This theorem together with Theorem \ref{thm;continuous} allow us to give a sufficient  condition 
for the nonzero variance. 

\begin{de}
	A function $\phi\in \widehat C_c^\infty(X)$ is said to be  dynamically null  with respect to   $( X, F)$ if for any  $F$-invariant  probability measure $\til \mu$ on $X$ one has 
	$\til\mu (\phi)= 0$. 
\end{de}
 
     \begin{thm}
     	\label{thm;checkable}
     	Let the notation and assumptions  be as in Theorem \ref{thm;central h}.
    If 
     	the variance $\sigma(F, \phi)$  in (\ref{eq;variance 0}) is zero, then the function 
$\phi$ is dynamically null with respect to $(X, F)$. 
     \end{thm}
 
 The above theorem says that the central limit theorem is nondegenerate if 
 the integral of $\phi$ with respect to an
 $F$-invariant probability measure   is nonzero.

\textbf{Acknowledgements:}
I would like to thank  Seonhee Lim, Weixiao Shen and Jiagang Yang for the discussions related to this work.

\section{preliminary}
\label{sec;review}

In this section we review some facts   and prove a couple of auxiliary results.
Let the notation and assumptions  be as in Theorem \ref{thm;central h}.
 Some results in this section are also used  in the proof of
Theorem \ref{thm;continuous} where we take  $H=G_a'$. 

We first  state the explicit description of the expanding cone and the result of effective multiple 
 correlations in \cite{shi}. 
 Let $\mathfrak h, \mathfrak u$ and $\mathfrak a $ be the Lie algebras of $H, U$ and $ A$,  respectively.
Let $\Ad: H \to \GL(\mathfrak h)$ and $\mathrm{ad}: \mathfrak h\to \mathrm{End}(\mathfrak h)$ be the adjoint representations of the  Lie group  and the Lie algebra, respectively.

Recall that we assume $(\star)$ holds, namely,  $a_1$ belongs to the expanding cone  $A_U^+ $ and the conjugation of  $a_1$ expands a horospherical subgroup $H_p^+$ contained in  $U$. The latter means that the eigenvalues of  $\mathrm{Ad}(a_1)$ on  the Lie algebra  of  $H_p^+$ are bigger than one.  The assumption $a_1\in A_U^+$
is also easy to check and the details are given below.

Let $\Phi(\mathfrak u)$ be the set of nonzero linear forms $\beta$ on $\mathfrak a$ such that there exists a nonzero $v\in \mathfrak u$ with 

\[
\mathrm{ad}(s) v={\beta(s)} v\quad \forall\  s\in \mathfrak a.  
\]
Let $B(\cdot, \cdot)$ be the Killing form on  $\mathfrak h$. 
Then for each $\beta \in \Phi(\mathfrak u)$, there exists a unique $s_\beta \in \mathfrak a$ such that 
\[
B(s_\beta, s)=\beta(s)\quad \forall\  s\in \mathfrak a. 
\]
\begin{thm}\cite{shi}
	\label{thm;cone}
	The expanding cone $A_U^+= \exp \mathfrak a_{\mathfrak u}^+$ where 
	$\mathfrak a_{\mathfrak u}^+ = \{\sum_{\beta\in \Phi(\mathfrak u)} t_\beta h_\beta: t_\beta>0 \}$.
\end{thm}

In the statement of the  effective multiple correlations we need to  use the
$(2, \ell)$-Sobolev norm 
for a positive integer $\ell$.
Recall that a vector field on a smooth manifold  $Y$ is a smooth section of the tangent bundle of $Y$. 
A vector field $v$ on $Y$ defines a partial differential operator $\partial^ v$ on 
  $C^\infty (Y)$. 
 Give 
$\alpha=(v_1, \ldots, v_k)$ where $v_i \ (1\le i\le k)$ are vector fields on $Y$ we take  $\partial ^\alpha=\partial ^{v_1}\cdots \partial ^{v_k}$ and $|\alpha|=k$.
We also allow $\alpha$ to be the null set  where $\partial ^\alpha$ is 
the identity operator and $|\alpha|=0$. 

Elements of the Lie algebra
$\mathfrak g$ are naturally identified with the right invariant vector fields on $G$
 which descend naturally
to vector fields on $X$.
For every  $v\in \mathfrak g$ 
 we use the same notation  $v$ for the induced vector field on $X$.  
We fix a basis $\mathfrak b$ of $\mathfrak g$ consisting of eigenvectors of $\Ad (a_1)$ and denote 
\[
\|\phi\|_\ell =\max_{|\alpha|\le \ell} \|\partial^\alpha \phi \|_{L^2_\mu},
\]
where the maximum is taken over all the  $k$-tuples $\alpha$ ($0\le k\le \ell $) with alphabet $\mathfrak b$.

We fix an inner product on $\mathfrak g$ such that elements of $\mathfrak b$ 
are orthogonal to each other. 
Let $\mathrm{d}_G$ be   the  Riemannian distance on  
$G$ given by a right invariant Riemannian manifold structure induced from the inner product on $\mathfrak g$.
The advantage of $\mathrm d_G$  is that there exists $\kappa >0$ such that 
\begin{align}
\label{eq;joy}
\mathrm d_G (1_G, a_t g a_{-t})\le e^{\kappa |t|} \mathrm d_G (1_G,  g )
\end{align}
for any $t\in \R$ and $g\in G$. 
The Lipschitz norm on $\widehat C_c^\infty(X)$ is defined by
\begin{align*}
\|\phi \|_{\mathrm{Lip}}&=\sup _{g,h\in G , g\Gamma\neq h\Gamma}\frac{|\phi(g\Gamma)-\phi(h\Gamma)|}{\mathrm{d}(g\Gamma, h\Gamma)}
\quad  \mbox{ where }
\quad 
\mathrm{d}(g\Gamma,h\Gamma)=\inf_{\gamma \in \Gamma} {\mathrm{d}_G(g\gamma, h)}. 
\end{align*}
The sup norm on 
  $\widehat C_c^\infty(X)$
is  defined by 
\[
\|\phi\|_{\sup}=\sup_{x\in X}|\phi(x)|.
\]

Recall that $\nu=\nu_{q, z}$ is a fixed probability measure on $X$ given by 
$z\in X$ and  $q\in C_c^\infty(U)$ through the formula (\ref{eq;measure}). 
\begin{thm}[\cite{shi} Thm.~4.5]\label{thm;correlation}
	There exists an absolute constant  $\delta>0$,  $\ell \in \N $ and $M\ge 1$ with the following properties: for any positive integer $k$,  any $\phi_i\in\widehat  C_c^\infty(X)\ (1\le i\le k)$ and any   real numbers $ t_1, t_2 ,\cdots ,t_k$ one has 
	\begin{equation}\label{eq;many}
	\begin{aligned}
	&\left| \int_X \prod_{i=1}^k\phi_i (a_{t_i} x) \dd \nu(x)-\int_X\phi_k\dd\mu\cdot
	\int_X \prod_{i=1}^{k-1}\phi_i (a_{t_i} x) \dd \nu(x)
	\right|      \\
	\le  & Mk \cdot  \max_{1\le i\le k} \|\phi_i\| \cdot (\max_{1\le i\le k} \|\phi_i \|_{\sup})^{k-1}\cdot
	e^{-\delta\min\{t_k, t_k-t_1, t_k-t_2, \ldots, t_k-t_{k-1}  \}},
	\end{aligned}
	\end{equation}	
where $\|\cdot \|$ is the norm on $\widehat C_c^\infty (X)$ given by 
\[
\|\phi\|=\max \{ \|\phi\|_{\sup}, \|\phi\|_\ell, \|\phi\|_{\mathrm{Lip}}    \}\quad \forall\  \phi\in \widehat C_c^\infty (X). 
\]

\end{thm}

We make it convention in this paper that the product indexed by  the null set is $1$ and the sum indexed by  the null set is $0$. 
 So in (\ref{eq;many}), if $k=1$ then $\prod_{i=1}^{k-1} \phi_i(a_{t_i}x)=1$.

\begin{lem}
	\label{lem;mixing}
 There exists $\delta'>0, E_0>0$ and $\ell_0\in \N$ such that for any functions $\phi, \psi\in \widehat C_c^\infty(X)$ and any $t\in \R$ one has 
	\[
	\left| \int_X \phi(a_t x)\psi (x)\dd\mu(x)-\int_X \phi \dd \mu \int _X \psi \dd \mu \right| 
	\le E_0\| \phi\|_{\ell_0} \|\psi  \|_{\ell_0} e^{-\delta ' |t|}
	\] 
\end{lem}
\begin{proof}
	Recall that we assume the action of $H$ on $X$ has a spectral gap and 
	 $F$ has nontrivial   projection to each simple factor of  $H$.
	 So
 the conclusion  follows from \cite[\S 6.2.2]{emv}. 
\end{proof}

Actually  the dynamical system $(X, \mu, a_t)$ is mixing of all orders. 
This fact is proved in \cite[Thm.~1.1]{beg} under the assumption of 
strong spectral gap. It might be possible to get a proof  from  \cite{beg}
 under   our weaker assumption, but we give a simpler proof here using the method of  \cite{shi}
for the completeness. The exponential  mixing  of all orders will follow from an estimate similar to  (\ref{eq;many}) with $\nu$
replaced by $\mu$. During the proof we will use the notation $f _1\lesssim_* f_2$ for two nonnegative functions
which  means 
$f_1\le C f_2$ for some positive constant $C$ possibly depending on $*$.

\begin{lem}
	\label{lem;orders}
	The conclusion of Theorem \ref{thm;correlation} holds with $\nu$ replaced by $\mu$. 
\end{lem}
\begin{proof}
	[Sketch of Proof] The proof is the same as that of \cite[Thm.~4.5]{shi}, so we only give a sketch of the key steps. The main difference is that  instead of using  quantitative nonescape of mass \cite[Thm.~1.3]{shi} we need to use the effective estimate of the volume in the cusp. 
	
	 We use $B_r^G$ to denote the open ball of radius $r$ in 
	$G$
	centered at the the identity element. 
	The injectivity radius  
	at $x\in X$ is defined by 
	\[
	I (x, X)=\sup \{ r>0: g\to gx \mbox{ is injective on } B_r^G   \}. 
	\]
For $\varepsilon>0$, let  $\mathrm{Inj}_\varepsilon=\{ x\in  X: I(x, X)\ge \varepsilon \}$.
We claim that 
	there exist $C_1, \delta_1>0$ such that 
	\begin{align}\label{eq;swim}
	\mu(X\setminus  \mathrm{Inj}_\varepsilon )\le C_1 \varepsilon^{\delta_1} \quad \mbox{for all } \varepsilon>0.  
	\end{align}
We don't have a direct proof of this fact, but it is a corollary   of  the equidistribution of measures in 
\cite[Thm.~1.4]{shi} and the quantitative nonescape of mass in \cite[Thm.~1.3]{shi}. 	More precisely, we fix a Haar measure 
$\mu_1$ on $G^+$ such that the open ball of radius $1$ in $G^+$ (denoted by $B^+_1$) has measure 
$1$. Since the action of $H$ on $X$ has a spectral gap, there exists $x\in X$ such that $Hx$ is dense in $X$. 
Therefore, by \cite[Thm.~1.4]{shi}  we have for any $\varphi \in C_c(X)$, 
\begin{align}\label{eq;ben}
\lim_{t\to \infty}\int_{B_1^+} \varphi(a_t h x)\dd \mu_1(h)= \mu(\varphi). 
\end{align}
On the other hand, by \cite[Thm.~1.3]{shi}, there exists $C_1>0 $ and $\delta_1>0$ such that 
\begin{align}\label{eq;sile}
\mu_1(\{ h\in B_1^+: a_t h x \in X\setminus \mathrm{Inj}_\varepsilon  \})\le C_1 \varepsilon^{\delta_1} \quad 
\mbox{for all }  t\ge 0 \mbox{ and }\varepsilon>0. 
\end{align}
In view of (\ref{eq;ben}) and (\ref{eq;sile}), if $\varphi: X\to [0, 1]$ and $\supp(\varphi)\subset X\setminus \mathrm{Inj}_\varepsilon$, then $\mu(\varphi)\le  C_1 \varepsilon^{\delta_1}$.
Now (\ref{eq;swim}) follows from 
 taking a sequence of increasing functions $\varphi$ which converges to the characteristic function of $X\setminus\mathrm{Inj}_\varepsilon$. 

To prove the lemma, we assume without loss of generality that 
	\[
	k\ge 2  \quad \mbox{and}\quad t=\min\{t_k, t_k-t_1, \ldots, t_k-t_{k-1}\}>0, 
	\]
	since otherwise the conclusion is trivial. 
Let $\ell =\ell_0+\dim G$ where $\ell_0$ be as in Lemma \ref{lem;mixing}  and  $0<r<1$ which is given by 
 \cite[(4.9)]{shi}. There is a smooth function $\theta : G^+\to [0, \infty)$ such that 
$\supp (\theta )$ is contained in the ball of radius $r$ in $G^+$ and $\|\theta  \|_\ell \lesssim r^{-2\ell}$. 

 Let $G^+$ be the unstable horospherical subgroup of $a_1$ in $G$. Since $H$ is a normal 
	subgroup of $G$, one has  $G^+$ is a subgroup of $H$. 
	Then 
	\begin{equation}\label{eq;afternoon}
	\begin{aligned}
	\int_{X} \prod_{i=1}^k\phi_i(a_{t_i}x ) \dd \mu(x)&=\int_{G^+} \theta (g)\int_X \prod_{i=1}^k\phi_i(a_{t_i}x ) \dd \mu(x)
	\dd g\\
	&=\int_{G^+}\int_X \theta (g)\prod_{i=1}^k\phi_i(a_{t_i} a_{t-t_k} g  a_{t_k-t}x ) \dd \mu(x)
	\dd g\\
	&= \int_{G^+}\int_X \theta (g)\varphi_k(a_t g x) \prod_{i=1}^{k-1}  \phi_i(g_ia_{t_i }x ) \dd \mu(x)	\dd g,
	\end{aligned}
\end{equation}
	where $g_i=a_{t-t_k+t_i }g a_{t_k-t_i-t} $ satisfies $\mathrm{d}_{G}(1_G, g_i)\lesssim r$. 
    So 	
    \begin{align}\label{eq;qiaoke}
    |\phi_i(g_ia_{t_i }x )-\phi_i(a_{t_i}x)|\le \|\phi_i \|_{\lip} r. 
    \end{align}
    Using (\ref{eq;qiaoke}), we can effectively replace $\phi_i(g_ia_{t_i }x )$ on the right hand side of (\ref{eq;afternoon}) by $\phi_i(a_{t_i}x) $, and reduce the estimate of (\ref{eq;afternoon}) to the estimate of 
    \begin{align}\label{eq;wuran}
     \int_X\int_{G^+}\theta (g)\varphi_k(a_t g x) \dd g\prod_{i=1}^{k-1}  \phi_i(a_{t_i }x )     \dd\mu(x ).
    \end{align} 
    We have  an effective estimate  of $\int_{G^+}\theta (g)\varphi_k(a_t g x)\dd g $ 
    for  $x\in \mathrm{Inj}_\varepsilon$  using \cite[Lemma 4.3]{shi}. 
    We have an effective estimate of the volume of $X\setminus\mathrm{Inj}_\varepsilon$ using
    (\ref{eq;swim}). These two estimates together allow us to give an effective estimate of (\ref{eq;wuran}).
\end{proof}

In the proof of the central limit theorem, we need to know the growth of the  norm $\|\cdot \|$ of the functions $\phi\circ a_{t}(x)=\phi(a_t x)$
for  $t\ge 0$. Clearly, for all $\phi\in \widehat C_c^\infty(X)$ we have  $\|\phi\|_\mathrm{sup}=\|\phi\circ a_{t}\|$. For the other two norms we have the following lemma.
\begin{lem}
	\label{lem;growth}
	There exists $\kappa>0$ such that for all $\phi\in \widehat C_c^\infty (X)$ and $t\ge 0$ one has 
	\begin{align}\label{eq;straight}
	\|\phi\circ a_t\|_\ell \le e^{\kappa t}\|\phi\|_\ell\quad \mbox{and }\quad 
	\|\phi\circ a_t\|_{\mathrm{Lip}} \le e^{\kappa t}\|\phi\|_{\mathrm{Lip}}. 
	\end{align}
	Therefore, $\|\phi\circ a_t\|\le e^{\kappa t}\|\phi\|$. 
\end{lem}
\begin{proof}
	Let $v\in \mathfrak b$ which is considered as a right invariant vector field   and let $v_{1}$ be the  value of $v$ at the identity element. 
	 Let $\til \phi$ be the lift of $\phi$ to $G$. 
For $g\in G$, 
	\[
	\partial ^v (\phi\circ a_t) (g\Gamma)=  v_1 (\til \phi(a_t  x g))=v_1 (\til \phi(a_t x a_{-t} \cdot a_t g))=\partial ^{\Ad(a_t)v  }\phi (a_t g\Gamma)=e^{\kappa_1 t} \partial ^v \phi(a_t g \Gamma), 
	\]
	where   we use the  assumption  that $\mathfrak b$ consists of eigenvectors of $\Ad(a_1)$
	so that there exists $\kappa _1\in \R$ with $\Ad(a_t)v=e^{\kappa_1 t}v$. In general for any $\alpha\in \mathcal  B^k$ there exists 
	$\kappa_\alpha\in \R$ such that 
	\[
	\partial ^\alpha(\phi \circ a_t)=e^{\kappa _\alpha} \cdot (\partial ^\alpha\phi)\circ a_t.
	\]
	Therefore the first inequality of (\ref{eq;straight}) holds for any $\kappa\ge \max_{|\alpha|\le \ell} \kappa_\alpha $. 
	
	For $g, h\in G$ with $g\Gamma\neq h\Gamma$, we have 
	\[
	\frac{|\phi(a_t g \Gamma)-\phi(a_t h\Gamma)|}{\mathrm{d}(g\Gamma, h\Gamma) }=
	\frac{|\phi(a_t g \Gamma)-\phi(a_t h\Gamma)|}{\mathrm{d}(a_tg\Gamma, a_th\Gamma) }
	\cdot 
	\frac{\mathrm{d}(a_tg\Gamma, a_th\Gamma)}{\mathrm{d}(g\Gamma, h\Gamma)}\le 
	\|\phi\|_{\lip}
	\frac{\mathrm{d}(a_tg\Gamma, a_th\Gamma)}{\mathrm{d}(g\Gamma, h\Gamma)}. 
	\]
	Suppose $\mathrm{d}(g\Gamma, h\Gamma)=\mathrm{d}_G(g, h\gamma)$ for 
	some $\gamma\in \Gamma$. Then by  the right invariance of $\mathrm{d}_G$, the definition of $\mathrm{d}$ and (\ref{eq;joy}),   we have
	\[
	\mathrm{d}(a_tg\Gamma, a_th\Gamma)\le \mathrm{d}_G(1_G, a_t h\gamma g^{-1}a_{-t})\le 
	e^{\kappa t}\mathrm{d}_G(1_G,  h\gamma g^{-1})= e^{\kappa t}{\mathrm{d}(g\Gamma, h\Gamma)}.
	\]
	Therefore the second inequality of (\ref{eq;straight}) holds. 	
\end{proof}

\begin{lem}
	\label{lem;condition}
	For any sequence $\{\phi_n \}$ of functions in  $\widehat C_c^\infty (X)$ one has
	\begin{align}\label{eq;sobolevp}
	\| \phi_1\cdots \phi _n\|\le n^\ell \cdot
	(\max_{1\le i\le n} \|\phi_i \|)^\ell\cdot
	(\max_{1\le i\le n}\|\phi_i\|_{\mathrm{sup}} )^{n-\ell}
	.
	\end{align}
\end{lem}

\begin{proof}
	
	It suffices to   show that (\ref{eq;sobolevp}) holds if we replace the norm in the left hand side by  any of  the three norms used to define $\|\cdot\|$. It is clear that this is true  
	for $\|\cdot\|_\mathrm{sup}$. 
	
	For different $x, y \in X$, we have 
	\begin{align*}
	&\frac{|\prod_{i=1}^ n\phi_{i}(x) -\prod_{i=1}^n\phi_{i}(y)|}{\mathrm{d}(x, y)}\\
	\le& 
	\frac{\sum_{k=1}^n |\prod_{i=1}^{k}\phi_{i}(x) \prod_{j={k+1}}^{n}\phi_{j}(y) -\prod_{i=1}^{k-1}\phi_{i}(x)  \prod_{i=k}^{n}\phi_{j}(y)|}{\mathrm{d}(x, y)}\\
	\le &  n\max _{1\le i\le n}\big( \|\phi_{i}\|_{\mathrm{Lip}}  \prod _{1\le j\le n, j\neq i}\|\phi_{j} \|_{\sup}\big).     
	\end{align*}
	Therefore $	\| \phi_1\cdots \phi _n\|_{\lip}\le  $ RHS of (\ref{eq;sobolevp}).
	
	Let    $\alpha=(v_1, \ldots, v_k) \in \mathfrak b^k$ where $1\le k\le \ell$.
	The product rule of the differential operators implies
	\begin{align*}
	\partial ^\alpha \prod _{1\le i\le n}\phi_{i}= \sum
	\prod_{i=1}^n\partial ^{\alpha_i}\phi_{i}, 
	\end{align*}
	where there are $n^k$ terms in the summation and at most $k\le \ell $ of the $\partial ^{\alpha_i}$ are are not the identity operator.  
	Therefore, 
	\begin{align*}
	\|\partial ^\alpha\prod _{1\le i\le n}\phi_{i}\|_{L^2_\mu} \le n^\ell 
	(\max_{1\le i\le n} \|\phi_{i} \|)^\ell
	(\max_{1\le m\le n} \|\phi_{m}\|_{\mathrm{sup}} )^{n-\ell}
	\le \mbox{RHS of (\ref{eq;sobolevp})},
	\end{align*}	
which completes the proof. 
	
\end{proof}

\section{Variance}\label{sec;variance}

In this section we prove  the finiteness of the  variance in (\ref{eq;variance 0}) and 
Theorem \ref{thm;variance}. All of them are contained in the following lemma. 

\begin{lem}
	\label{lem;variance}
	Let the notation and the assumptions be as in Theorem \ref{thm;central h}. 
	In particular $\phi\in \widehat C_c^\infty (X)$ and $\mu(\phi)=0$.
	For a real number $t\in \R$, let $\phi_t$ be the function $\phi(a_tx)$. 
	Then 
	\begin{enumerate}[\rm (i)]
		\item $\lim_{l\to \infty} 2\int_0^l \mu(\phi_t\phi)\dd t$
		converges to a nonnegative real number, i.e., $\sigma(\phi, F)$ of (\ref{eq;variance 0}) is  well-defined and  nonnegative. 				
			\item   
			\begin{align}\label{eq;variance}
		\lim\limits_{T\to \infty}\frac{1}{T} \int_0^T\int_0^T \int_X \phi_t\phi_s\dd \nu\dd t \dd s =\sigma(\phi, F).
			\end{align} 
		\item $\sigma(\phi, F)=0$ if and only if there exists $\varphi\in L^2_\mu$ such that 
		for all $s> 0$
		\begin{align}\label{eq;cohomology}
		\int_0^s \phi_t(x)\dd t = \varphi(a_sx)-\varphi(x)  \qquad \mbox{for } \mu \mbox{-a.e. }  x\in X.   
		\end{align}
	
	\end{enumerate}

\end{lem}

\begin{proof}
	(\rmnum{1})
    In view of  Lemmas \ref{lem;mixing},   there exists $\delta'>0$ 
	such that  $|\mu(\phi_t\phi)|\lesssim_\phi
	e^{-\delta' t}$ for  all $t\ge 0$.
	Therefore, $2\int_0^l \mu(\phi_t \phi) \dd t$ converges as $l\to \infty$.
	This proves that $\sigma(\phi, F)$ is well-defined. 
	
	 To prove $\sigma(\phi, F)$  is nonnegative, we  show that it is equal to 
	\begin{align*}
\sigma'\df	\lim_{T\to \infty} 
\int_X 	\frac{1}{T}\left  (\int_0^T  \phi _t \dd t \right )^2 \dd\mu=
 	\lim_{T\to \infty} 	\frac{1}{T}\int_0^T \int_0^T \mu (\phi_t \phi_s)  \dd t \dd s,
 	\end{align*}
 	which is obviously nonnegative. 
 By the  symmetry of $t$ and $ s$, the invariance of $\mu$ under $F$, we have
 	\[
 	\sigma'=\lim_{T\to \infty} 	\frac{2}{T}\int_0^T \int_s^T \mu(\phi_{t} \phi_s)\dd t \dd s= \lim_{T\to \infty} 	\frac{2}{T}\int_0^T \int_s^T \mu(\phi_{t-s}\phi)\dd t \dd s.
 	\]
 	We make change of variables $(s, t)\to (s, r)= (s, t-s)$, then
 	\begin{align*}
  \sigma'&=\lim_{T\to \infty} 	\frac{2}{T}\int_0^T \int_0^{T-r} \mu(\phi_{r}\phi)\dd s \dd r
  =\lim_{T\to \infty}	\frac{2}{T} \int_0^T  (T-r) \mu(\phi_{r}\phi) \dd r\\
& =	\sigma( \phi, F)-\lim_{T\to \infty} \frac{2}{T} \int_0^T r\mu(\phi_{r}\phi)\dd r.
 	\end{align*}
 	Since $|\mu(\phi_{r}\phi)|\lesssim_{\phi} e^{-\delta 'r}$ for  $r\ge 0$,  the value $ |\int_0^T r\mu(\phi_{r}\phi)\dd r|$ is uniformly bounded for all $T>0$. 
 	So $\lim_{T\to \infty} \frac{2}{T} \int_0^T r\mu(\phi_{r}\phi)\dd r=0$, and hence 
 $\sigma( \phi, F)=\sigma'$.

(ii) 
Suppose  $l\ge 1$  and $T\ge 10 l$. 
 For every $r$ with $0\le r\le l$, 
 we apply (\ref{eq;many}) for $k=1$,  the function $\phi \phi _{r }$ and $s\ge 0$, then  
   we have
\begin{align}\label{eq;made}
\int_X \phi_s \phi _{r+s }\dd  \nu = \mu(\phi \phi_r)+O_{\phi, l}(e^{-\delta  s}). 
\end{align}
We  decompose   the function 
on $X$ defined by 
\begin{align}\label{eq;21}
\frac{1}{2}\int_0^T\int_0^T \phi_s( x )\phi_t(  x)\dd t \dd s =\int_0^T\int_s^T \phi_s( x )\phi_t(  x)\dd t \dd s 
\end{align}
into $\xi(x)+\eta(x)$, where 
\begin{align*}
\xi(x)&= \int_0^T\int_s^{\min \{T, s+l \} }\phi_s(x) \phi_t(  x )\dd t \dd s,	 \\
\eta(x)&= \int_0^{T-l}\int_{s+l}^{T }\phi_s( x) \phi_t( x ) \dd t \dd s.
\end{align*}

To calculate $\nu(\xi)$ we make change of variables $(s, t)\to (s, r)= (s, t-s)$ and apply Fubini's theorem:
\begin{equation}
\label{eq;22}
\begin{aligned}
\nu(\xi)&=\int_0^l\int_0^{T-r }\int_X \phi_s\phi_{r+s}\dd\nu 
\dd s \dd r
\\
 (\mbox{by (\ref{eq;made})})\qquad &=\int_0^l\int_0^{T-r }\mu (\phi_{r}  \phi )+O_{ \phi, l}(e^{-\delta s}) \dd s\dd r \\
&=T\int_0^l   \mu (\phi_{r}  \phi ) \dd r+O_{\phi, l, \delta }(1). 
\end{aligned}
\end{equation}
On the other hand, by (\ref{eq;many}) with $k=2$,   $t_2=t$ and $t_1=s$, 
\begin{equation}
\label{eq;23}
\begin{aligned}
\nu(\eta)&=\int_0^{T-l} \int_{s+l}^T   \int _X
\phi_t\phi_s \dd\nu \dd t \dd s\\
&=\int_0^{T-l} \int_{s+l}^T  O_{\phi}(e^{-\delta (t-s)})\dd t\dd s\\
&=\int_0^{T-l} O_{\phi, \delta} (e^{-\delta l})\dd s\\
&=(T-l) O_{\phi, \delta } (e^{-\delta l}).
\end{aligned}
\end{equation}
By (\ref{eq;21}), (\ref{eq;22}) and (\ref{eq;23}), we have 
\begin{equation}\label{eq;final}
\begin{aligned}
\frac{1}{T} \int_0^T\int_0^T \int_X \phi_t( x )\phi_s( x)\dd \nu\dd t \dd s 
=\frac{2\nu(\xi)+2\nu(\eta)}{T} \\
=2\int_0^l   \mu (\phi_{r}  \phi ) \dd r+\frac{1}{T}O_{\phi, l, \delta}(1)+O_{\phi,\delta} (e^{-\delta l}).
\end{aligned}
\end{equation}

We show that the left hand side of  (\ref{eq;final}) is  arbitrarily close to $\sigma(\phi, F)$ provided
that  $T$ is sufficiently large. Given $\varepsilon>0$, in view of (\rmnum{1}), there exists $l\ge 1$ such that 
  \begin{align*}
  \left |2\int_0^l   \mu (\phi_{r}  \phi ) \dd r-\sigma(\phi, F)\right |< \varepsilon\quad \mbox{   and }
  \quad 
  |O_{\phi,\delta} (e^{-\delta l})|<\varepsilon.
  \end{align*} 
  For this fixed 
  $l$ we have  $|\frac{1}{T}O_{\phi, l, \delta}(1)|<\varepsilon$ provided that  $T$ is sufficiently large. 
  Therefore  		
(\ref{eq;variance}) holds.

(iii)
 Suppose (\ref{eq;cohomology}) holds for all $s> 0$. Recall that we have proved in (\rmnum{1}) that 
\begin{align}\label{eq;party}
\sigma( \phi, F)=\lim_{T\to \infty} 
\int_X 	\frac{1}{T}\left  (\int_0^T  \phi _t \dd t \right )^2 \dd\mu.
\end{align}
By
(\ref{eq;cohomology}), $\int_0^T  \phi _t(x) \dd t=\varphi(a_T x)-\varphi(x)$ for $\mu$-a.e.~$x$.
This together with  (\ref{eq;party}) and $\varphi\in L^2_\mu$ imply  
 $\sigma(\phi, F)=0$. 

Now we assume $\sigma( \phi, F)=0$. 
We claim  that \begin{center}
the $L^2$-norm of  $\xi_T( x)\df \int_0^T \phi_t(x)\dd t$
 is uniformly bounded for all $T\ge 0$.
\end{center}
Let $\eta(x)=\xi_1( x)$ and $\eta_i(x)=\eta(a_{i}x)$. Then the claim    is equivalent  to the $L^2$-norm of 
 $\xi_n( x)=\sum_{i=0}^{n-1} \eta_i$    is uniformly bounded for all $n\ge 2$.

 Since $\mu$ is $F$-invariant, 
 \begin{align}\label{eq;multiple}
  \int_X  (\sum_{i=0}^{n-1}\eta_i)^2\dd\mu =n \mu(\eta)^2 +2\sum_{i=1}^{n-1}\sum_{j=0}^{i-1}\mu(\eta_{i-j} \eta)=n \mu(\eta)^2 +2\sum_{i=1}^{n-1}(n-i)\mu(\eta_{i} \eta).
 \end{align}
   By (\ref{eq;party}) and (\ref{eq;multiple})
 \[
0= \sigma( \phi, F)=\lim_{m\to \infty} \int_X\frac{1}{m} \big (\sum_{i=0}^{m-1} \eta_i\big) ^2 \dd\mu= \mu(\eta^2)+\lim_{m\to \infty}2\sum_{i=1}^{m-1}\big (1-\frac{i}{m}\big)\mu(\eta_{i} \eta). 
 \]
We solve  $\mu(\eta^2) $ from the above equation  and plug it in (\ref{eq;multiple}), then 
\begin{align}
 \int_X  (\sum_{i=0}^{n-1}\eta_i)^2\dd\mu 
   & =
    -2n
   \left( \lim_{m\to \infty }\sum_{i=1}^{m-1}\big(1-\frac{i}{m}\big)\mu(\eta_{i} \eta)\right)+2\sum_{i=1}^{n-1}(n-i)\mu(\eta_{i} \eta)
\notag\\
   &= -2n
   \left( \lim_{m\to \infty }\left(\sum_{i=1}^{n-1}+\sum_{i=n}^{m-1}\right)\big(1-\frac{i}{m}\big)\mu(\eta_{i} \eta)\right)+2\sum_{i=1}^{n-1}(n-i)\mu(\eta_{i} \eta)
 \notag  \\
   &=-2n \lim_{m\to \infty }\sum_{i=n}^{m-1} \big(1-\frac{i}{m}\big) \mu(\eta_{i} \eta)-2\sum_{i=1}^{n-1}i\mu(\eta_{i} \eta).  \label{eq;strong}
 \end{align}

   Note that the function $\eta\in \widehat C_c^\infty(X)$.
   So Lemma \ref{lem;mixing} implies 
    \[
   \left  |2n\lim_{m\to \infty}\sum_{i=n}^{m-1}\big (1-\frac{i}{m}\big) \mu(\eta_i \eta) \right |\lesssim_\eta n  \sum_{i=n}^{m-1} e^{-\delta' i}\le \frac{ne^{-\delta' n}}{1-e^{-\delta' }},
    \]
   which  converges to zero as $n\to \infty$.  So the absolute value of the   first  term of (\ref{eq;strong}) is uniformly bounded for all $n\in \N$. 
 Similarly, we can uniformly  bound      the second term of (\ref{eq;strong}) as
   \[
   \left | 2
   \sum_{i=1}^{n-1}i\mu(\eta_i \eta)\right |\lesssim _\eta  \sum_{i=1}^{n-1} ie^{-{\delta' i}}\lesssim
   \int_1^\infty t e^{-\delta' t} \dd t < \infty.
   \]
    Therefore,  the $L^2$-norm of $\xi_n$  is uniformly bounded for all $n\in \N$ and the proof of the claim is complete. 
    
 Note that the Hilbert space $L^2_\mu$ is self-dual.   So the claim and  the Alaoglu's theorem   imply that there exists a subsequence $\{n_i \}$ of natural numbers  and $\varphi\in L^2_\mu$ such that 
     $\lim_{i\to \infty}\xi_{n_i}=-\varphi$  in the  weak$^*$ topology.
        We show that (\ref{eq;cohomology}) holds for this  
    $\varphi$. 
 It is not hard to see that   the function  $\varphi(a_s x)$ is a weak$^*$ limit  of  the sequence 
    $\{-\xi_{n_i}( a_s x)\}_i$ in $L^2_\mu$.
    Therefore, 
    in the weak$^*$ topology we have 
    \begin{align}\label{eq;guaiguai}
   \varphi\circ a_s- \varphi =\lim_{i\to \infty}\big (\xi_{n_i}-\xi_{n_i}\circ a_s\big )=
    \int_0^s \phi_t \dd t-\lim_{i\to \infty} \int_0^s \phi_{t+n_i}\dd t.
    \end{align}
 On the other hand,
 given any $\psi\in C_c^\infty(X)$,  by Lemma \ref{lem;mixing}
 \begin{align*}
\left |\lim_{i\to \infty}\int_0^s \int_X\phi_{t+n_i}(x)\psi(x) \dd\mu(x)\dd t \right |  
&\le \lim_{i\to \infty}\int_0^s  \left | \int_X\phi_{t+n_i}(x)\psi(x)\dd \mu(x) \right  | \dd t\\
&\lesssim_{\phi, \psi}\lim_{i\to \infty}
\int_0^s e^{-\delta' (t+n_i)}\dd t=0.
 \end{align*}
Therefore   $\lim_{i\to \infty} \int_0^s \phi_{t+n_i}(x)\dd t=0$ in the weak$^*$
    topology. 
    This observation together with (\ref{eq;guaiguai}) imply (\ref{eq;cohomology}). 
\end{proof}

 \section{Proof of the  central limit theorem}\label{sec;central}
 
Let the notation and the assumptions be as in Theorem \ref{thm;central h}. 
In particular,  we fix a function  $\phi\in \widehat C_c^\infty(X)$ and  take $\phi_t=\phi\circ a_t$. 
 In  this section the dependence of constants on $\phi$ and the probability measure
 $\nu$ will not be specified.  By possibly replacing  $\phi $ by $\frac{\phi}{N}$ we assume without loss of 
generality that 
 $\|\phi\|_{\mathrm{sup}}\le \|\phi\|\le 1$. 
We  assume that  $\phi$ is not identically zero, since otherwise the conclusion holds trivially. 
Let $M, \ell, \kappa\ge 1$ and $\delta>0$ so that Theorem \ref{thm;correlation}  and 
Lemma \ref{lem;growth} hold.

We need to show that 	the random variables 
\[
S_T: (X, \nu )\to \R \quad \mbox{where }\quad  S_T(x)=\frac{1}{\sqrt T}\int_0^T \phi_t(x)  \dd t 
\]
converges  as $T\to \infty$ to the normal distribution with mean zero and  variance 
$\sigma=\sigma(\phi, F)$.  
Our main tool is the following so called the second limit theorem in the theory of probability. 
\begin{thm}[\cite{fs}]
	\label{thm;sencond}
	If for every $n\ge 0$ \begin{equation}\label{eq;moment}
	\lim_{T\to \infty}\frac{1}{n!} \int_X  S_T^n (x)\dd \nu(x)=
	\left\{
	\begin{array}{ll}
		\frac{({\sigma}/{2})^{{n}/{2}}}{({n}/{2})!} & \mbox{for  } n \mbox{ even }\\
			0 & \mbox{for  } n \mbox{ odd }\\
	\end{array}
	\right. , 
	\end{equation}
then as $T\to \infty$ the distribution of random variables $S_T$ on $(X, \nu)$ converges  to the  normal distribution with mean zero and variance $\sigma$.  
\end{thm}

Note that (\ref{eq;moment}) obviously  holds for $n=0$ and $n=1$.  The case of $n=2$
is proved in Lemma \ref{lem;variance}(\rmnum{2}). 
So we assume   $n\ge 3$ in the rest of this section unless otherwise stated.  
 We will 
 estimate  
\begin{align*}
b_{n, T}=
\frac{1}{n!} 
\int_X
\left(\int_0^T \phi_t(x)\dd t\right)^n\dd \nu
\end{align*}  
for each fixed $n\ge 3$. 
Write $\bt=(t_1, \ldots, t_n)$,  $
\dd \bt = \dd t_1 \cdots \dd t_n$, $\phi_\bt=\prod_{i=1}^n \phi_{t_i}$   and 
\[
[0,T]^n_{\le  }=\{\bt\in [0, T]^n: t_1\le t_2\le \cdots\le t_n   \}.
\] 
By the symmetry of the  variables of $\bt$ and the  Fubini's theorem,  
\begin{align}\label{eq;reduction}
b_{n, T}=
\int_X
\int_{[0,T]^n_{\le }} \phi_\bt(x)\dd \bt\dd \nu=\int_{[0,T]^n_{\le }}\int_X\phi_\bt(x)
\dd \nu\dd \bt.
\end{align}

An ordered partition   $$P=\{ \{1, 2,\ldots, m_1 \}, \{m_1+1, \ldots, m_2 \}, \ldots , \{m_{|P|-1}+1
,\ldots, n-1, n\}   \}  $$ of $\{1, \ldots,n   \}$ 
determines 
$|P|-1$ positive integers   $1\le  m_1<\cdots< m_{|P|-1}< n$  and vice versa.
 We use $P\prec  n$ to denote $P$ is an ordered partition of $\{1, \ldots, n \}$. 
Let $ |P|=k$ be the cardinality of 
$P$ and let 
 $m_0=0, m_{k}=n$. 
Although $k$ and  $m_i$  depend on $P$, we will not specify it for simplicity. 
We fix a positive  real number   
\begin{align}\label{eq;b}
b= n+\frac{4\kappa\ell +2}{\delta} 
\end{align}
and set
\begin{align}\label{eq;IP}
I_P=\{\bt\in [0, T]^n_\le :& t_{m_i+1}-t_{m_i}>b^{m_{i+1}-m_i}n\log T \mbox{ for }  1\le i < k   \\
&\mbox{ and } t_{j+1}-t_{j}\le b^{m_{i}-j}n\log T \mbox{ for } m_{i-1} <j< m_i, 1\le i\le k \}.\notag
\end{align}
We will always assume 
\begin{align}\label{eq;t}
T> Mn^2 b^n \log T
\end{align}
so that $I_P\neq \emptyset $
for any $P\prec n$. 
This will allow us to  avoid ambiguity in the discussions below.

\begin{lem}
	\label{lem;disjoint}
The set 	 $ [0, T]^n_\le $ is a disjoint union of $I_P $ where  $P$ is taken over  all the ordered  partitions of $\{1, \ldots,n \}$.
\end{lem}
\begin{proof}
Given $\bt \in [0, T]^n_\le$, we show that there exists a $P\prec n$ such that $\bt \in I_P$. 
We find the blocks of $P$ from the  top to the  bottom. 
Let \[
m=\max  \{1\le  j<n:  t_{j+1}-t_j  >b^{ n-j}n \log T  \}, 
\]
where we interpret $m=0$ if the set before taking the maximum is empty. This $m$ is $m_{k-1}$. To find other $m_i$ we do the same calculation for the set $\{1, 2, \ldots, m \}$. The process must terminates with  $m=0$ in finite steps and it gives a partition $P$ such that $\bt \in I_P$. It is not hard to see from the construction  that $P$ is uniquely determined by $\bt$. So $ [0, T]^n_\le $ is a disjoint union of $I_P$. 
\end{proof}

  In view of (\ref{eq;reduction}) and Lemma \ref{lem;disjoint}, we have 
\begin{align*}
b_{n, T}=\sum_{P\prec n}b_{n, T, P} \quad \mbox{where} \quad b_{n, T, P}=         \int_{I_P }\int_X
\phi _{\bt }(x) \dd \nu \dd \bt.
\end{align*}
Now we estimate $b_{n, T , P}$
for  each  fixed  partition $P$.
We write $\bt=(\bt_1, \ldots, \bt_k)$ according to the partition $P$, i.e.~$\bt_i=(t_{m_{i-1}+1},  \ldots, t_{m_{i}})$. For $1\le i\le k$ and $\bt \in [0, T]^n$  we let $$ 
r_{i , \bt}= t_{m_{i-1}+1}, s_{i, \bt }=t_{m_i}, 
\bt_i-r_{i, \bt}= (t_{m_{i-1}+1}-r_{i,\bt}, \ldots, t_{m_{i}}-r_{i, \bt }) \mbox{ and } 
\phi_{\bt_i}=\prod _{m_{i-1}< j \le m_{i}} \phi_{t_j}
.$$
Let  
\[
\br =(r_1, \ldots, r_k)\in \R^k \quad \mbox{and }\quad 
I_P(\br )=\{\bt\in I_P: r_{i, \bt}=r_i \}.
\]
 Let 
\[
R_P=\{\br \in [0, T]^k: I_P(\br )\neq \emptyset  \}, \dd \br =\prod_{i=1}^k \dd r_i\quad\mbox{and } \dd \bt_{P}=\prod_{i\in E_P }\dd t_i,
\]
where $E_P=\{1, \ldots, n \}\setminus \{ m_{i-1}+1: 1\le i\le k  \}$.
By slightly abuse of notation
 we write $\dd \bt=\dd \br \dd \bt_P$ by identifying $t_{m_{i-1}+1}$ with $r_i$ for $1\le i\le k$. 

\begin{lem}
	\label{lem;size}
	The volume of $I_P(\br )$ with respect to $\dd \bt_{P}$
	is at most $ b^{n^2}n^n\log^n T$. 
\end{lem}
\begin{proof}
	This is almost a trivial estimate by noting that if $ m_{i-1}<j< m_i $ for some 
	$1\le i\le k$, then 
	$t_{j+1}-t_{j}\le b^n n \log T$. 
\end{proof}

\begin{lem}
	\label{lem;volume}
One has 
	\begin{align}\label{eq;rp1}
	0\le \frac{T^{k}}{k!}-	\int_{R_P}1 \dd \br\le {n^3 b^n T^{k-1}\log T}.
	\end{align}
\end{lem}
\begin{proof}

Since $R_P$ is a subset of $[0, T]^k_\le $, one has 
\[
0\le \int_{[0, T]^k_\le}1\dd \br -\int_{R_P}1 \dd \br=\frac{T^k}{k!}-\int_{R_P}1 \dd \br,
\]	
which gives the lower bound in (\ref{eq;rp1}).

To prove the upper bound we need more precise information about $R_P$. 
Let 
\begin{align*}
b_{k-1}&=b^{m_k-m_{k-1}}n\log T, \\
b_{k-2}&=(b^{m_k-m_{k-1}}+b^{m_{k-1}-m_{k-2}})n\log T,\\
&\cdots \qquad \cdots \\	
b_1&=(b^{m_k-m_{k-1}}+\cdots+b^{m_2-m_1})n
\log T  .
\end{align*}
We set $r_0=0$,  $b_0=b_1$ and $b_{k}=0$.  We claim that 
\begin{align}\label{eq;RP1}
\int_{R_P }1 \dd \br =
\prod_{i=1}^k\int_{r_{i-1}+b_{i-1}-b_i}^{T-b_{i}}  1\dd r_i.
\end{align}
If $k=1$, then $R_P=[0, T]$ and (\ref{eq;RP1}) holds. 
Suppose $k\ge 2$. If $\br \in R_P$, then by definition $I_P(\br)\neq \emptyset$. So there exists $\bt \in I_P$ such that $r_{i, \bt}=r_i$. In view of the definition of $I_P$ in (\ref{eq;IP}), one has  
\begin{align}\label{eq;cold}
r_{i}>r_{i-1} + b_{i-1}-b_i\quad \forall 1<i\le k. 
\end{align}
On the other hand, if (\ref{eq;cold}) holds for some $\br$, then $I_P(\br) $
is nonempty and hence $\br\in R_P$. The formula (\ref{eq;RP1}) follows from 
(\ref{eq;cold}).

Now we prove the upper bound in  (\ref{eq;rp1}). 	
	Note that the length of   $$
	L_i\df [r_{i-1},T ]\setminus [r_{i-1}+b_{i-1}-b_i, T-b_i]$$
	is at most $b_{i-1}\le b^n n^2\log T$. So
	\begin{equation}
	\label{eq;RP2}
	\begin{aligned}
&	\prod_{i=1}^k\int_{r_{i-1}}^{T}  1\dd r_i-\prod_{i=1}^k\int_{r_{i-1}+b_{i-1}-b_i}^{T-b_i}  1
	\dd r_i  \\
	= & \sum_{j=1}^k\left( \prod_{i=1}^{j-1} \int_{r_{i-1}+b_{i-1}-b_i}^{T-b_i}   1\dd r_i\int_{L_j}1 \dd r_j
\prod _{s=j+1}^k \int_{r_{s-1}}^{T}1 \dd r_s \right)\\		
	\le  & \sum_{j=1}^k\left( \prod_{i=1}^{j-1} \int_{0}^{T}   1\dd r_i\int_{L_j}1 \dd r_j
\prod _{s=j+1}^k \int_{0}^{T}1 \dd r_s \right)\\		
	\le &k\cdot b^n n^2 \log T\cdot 
	T^{k-1}
	\le    b^n n^3 T^{k-1}\log T
	.
	\end{aligned} 	
	\end{equation}	
\end{proof}

\begin{cor}
	\label{cor;small}
	If $| P|=k< \frac{n}{2}$,  then
$ \lim_{T\to \infty}  \frac{b_{n, T, P}}{T^{\frac{n}{2}}}=0.
$
\end{cor}

\begin{proof}
Recall that we assume $\|\phi\|\le 1$. So by Lemmas \ref{lem;size} and \ref{lem;volume}
\begin{align*}
|b_{n, T, P}|\le \int_{R_P} \int_{I_P(\br)}1 \dd \bt _P \dd \br \le \frac{T^k}{k!}\cdot  b^{n^2}n^n \log^n   T. 
\end{align*}
The conclusion follows from the above estimate and the assumption that $k<\frac{n}{2}$. 
	
 \end{proof}

Now we estimate $b_{n,T, P}$ for $k=|P|\ge \frac{n}{2}> 1$ using Theorem \ref{thm;correlation}. 
By Lemma \ref{lem;growth}, 
$\|\phi_t\|\le e^{ \kappa t }\|\phi \|$ for $t\ge 0 $. 
So in view of the definition of 
$I_P$ in (\ref{eq;IP}) and the assumption $\|\phi \|\le 1$
\begin{equation}
\label{eq;estimate}
\begin{aligned}
\max_{m_{i-1}< j\le m_i} \|\phi_{t_j-r_{i, \bt} } \|&\le 
\exp (\kappa (t_{m_i}-r_{i, \bt}) ) \\
&=
\exp (\kappa \sum_{m_{i-1}<j<m_i} (t_{j+1}-t_j))\\
&\le \exp ((\kappa n\log T) \sum_{m_{i-1}<j<m_i}  b^{m_i-j} )\\
&\le \exp (2\kappa nb^{m_i-m_{i-1}-1}\log T ),
\end{aligned}
\end{equation}
where in the last estimate we use $b\ge n$. 
In (\ref{eq;estimate}), {we interpret $\sum_{m_{i-1}<j<m_i} (t_{j+1}-t_j)=0$ if there is no integer $j$ satisfying $m_{i-1}<j<m_i$.}
By Lemma  \ref{lem;condition},  the assumption $\|\phi\|_{\sup}\le \|\phi\| \le 1$
and (\ref{eq;estimate}), 
\begin{equation}\label{eq;combine}
\begin{aligned}
\|\phi_{\bt_i-r_{i, \bt}}\|&\le (m_{i}-m_{i-1})^\ell \cdot
(\max_{m_{i-1}< j\le m_i} \|\phi_{t_j-r_{i, \bt} } \|)^{\ell}\\
 & \le (m_i-m_{i-1})^\ell \cdot  \exp ({2\kappa n\ell b^{m_i-m_{i-1}-1}\log T })\\
 & \le n^\ell \exp ({2\kappa n\ell b^{m_i-m_{i-1}-1}\log T })\\
 & \le\exp ({4\kappa n\ell b^{m_i-m_{i-1}-1}\log T }),
\end{aligned}
\end{equation}
where in the last estimate we use $n\le e^n$.

For $1\le i\le k$, 
using  Theorem \ref{thm;correlation} with the product of $m_{i-1}+1$ functions
$\phi_{s} \ (1\le s\le m_{i-1})$ and $\phi_{\bt_i}$, we have
\begin{equation}\label{eq;jiaquan}
\begin{aligned}
&\left|\int_X \prod_{j=1}^{i}\phi_{\bt_j} \dd \nu-\int_{X}\prod_{j=1}^{i-1}\phi_{\bt_j}\dd\nu \int_X\phi_{\bt_i-r_{i, \bt}}\dd \mu\right| \\
=&\left| \int_X\phi_{t_1}\phi_{t_2}\cdots \phi_{t_{m_{i-1}}}\phi_{(\bt_i-r_{i, \bt})+r_{i, \bt}} \dd \nu-\int_{X}\phi_{t_1}\phi_{t_2}\cdots \phi_{t_{m_{i-1}}}\dd\nu \int_X\phi_{\bt_i-r_{i, \bt}}\dd \mu\right| \\
\le & M (m_{i-1}+1) \cdot\max \{\|\phi \|, \|\phi_{\bt_i-r_{i, \bt}} \|  \} 
e^{-\delta(r_{i, \bt}-s_{i-1, \bt}) }, 
\end{aligned}
\end{equation}
where we use $\|\phi \|_{\sup}\le 1$ and set $s_{0, \bt}=0$.
The definition of 
$I_P$ in (\ref{eq;IP}) implies 
\begin{align}\label{eq;0.07}
r_{i, \bt}-s_{i-1, \bt}>b^{(m_i-m_{i-1})}n\log T\quad \mbox{for }i\ge 2. 
\end{align}
By (\ref{eq;combine}), (\ref{eq;jiaquan})  and  (\ref{eq;0.07}) we have for $2\le i\le k$
 \begin{equation}\label{eq;ran}
 \begin{aligned}
& \left|\int_X \prod_{j=1}^{i}\phi_{\bt_j} \dd \nu-\int_{X}\prod_{j=1}^{i-1}\phi_{\bt_j}\dd\nu \int_X\phi_{\bt_i-r_{i, \bt}}\dd \mu\right| \\
 \le&  M n \exp( 4\kappa b^{(m_i-m_{i-1}-1)}\ell n\log T )\cdot \exp({ -\delta b^{(m_i-m_{i-1})}n\log T}) \\
 \le & {e^{-n \log T}}\le {T^{-n}} , 
 \end{aligned}
 \end{equation}
 where in the last line we use $b\ge \frac{4\kappa \ell+2}{\delta}$ in (\ref{eq;b}) and $T\ge Mn$ in (\ref{eq;t}). 
\begin{rem}
	\label{rem;simple}
	The estimate (\ref{eq;ran}) also holds for $i=1$ provided that $t_1\ge b^{m_1} n \log T$, which will be used later. 
\end{rem} 

By   (\ref{eq;ran}) and    the assumption $\|\phi\|_{\sup}\le 1 $ we have 
\begin{equation}\label{eq;main1}
\begin{aligned}
&\left | b_{n, T, P}-\int_{I_P}\left(\int_{X}\phi_{\bt_1}\dd\nu\right)\left( \prod_{i=2}^k \int_X \phi_{\bt_i-r_{i, \bt}}\dd \mu\right )\dd \bt\right | \\
\le   &\int_{I_P} \left| \int_X	 \prod_{i=1}^k\phi_{\bt_i} \dd \nu-\int_{X}\phi_{\bt_1}\dd\nu\prod_{i=2}^k \int_X\phi_{\bt_i-r_{i, \bt}}\dd \mu \right |
\dd  \bt  \\
\le   &\int_{I_P}\sum_{i=2}^k \left| \int_X	 \prod_{j=1}^{i}\phi_{\bt_j} \dd \nu-\int_{X}\prod_{j=1}^{i-1}\phi_{\bt_j}\dd\nu \int_X\phi_{\bt_i-r_{i, \bt}}\dd \mu \right | 
\dd  \bt  \\
\le & \int_{I_P}n T^{-n}\dd \bt \le 1.
\end{aligned}
\end{equation}
\begin{lem}
	\label{lem;separate}
	Suppose $n\ge 3$ and  
		 $k=| P| \ge  \frac{n+2}{2}$,  then
	$ \lim_{T\to \infty}  \frac{b_{n, T, P}}{T^{\frac{n}{2}}}=0.
	$
\end{lem}
\begin{proof}
	If $k>\frac{n}{2}$, then the partition  $P$ must contains a single number. 
	Since $k\ge \frac{n+2}{2}$, there exists $i>1$ such that $\{ m_i\}\in P$.  
	Therefore, 
	\[
	\prod_{i=2}^k \int_X \phi_{\bt_i-r_{i, \bt}}\dd \mu=0.
	\]
	This equality and (\ref{eq;main1}) implies 	$ \lim_{T\to \infty}  \frac{b_{n, T, P}}{T^{\frac{n}{2}}}=0.
	$
\end{proof}

Now there are two cases left, namely, 
\begin{align}
&n \mbox{ is odd  and } P=P_1=\{ \{1 \} , \{2, 3\} , \cdots,\{n-1, n \}  \},
\label{eq;odd}\\
&n \mbox{ is even  and } P=P_2=\{ \{1, 2 \} , \{3, 4\} , \cdots,\{n-1, n \}  \}.
\label{eq;even}
\end{align}

\begin{lem}
	\label{lem;equal} If $n\ge 3$ is even,   and $P=P_2$ is given by (\ref{eq;even}), then for $k=|P|=\frac{n}{2}$
	\begin{align}\label{eq;equal}
\lim_{T\to \infty}\frac{b_{n, T, P} }{T^k}	=\frac{\left(\frac{\sigma}{2}\right)^{k}}{k!}.  
	\end{align}
\end{lem}
\begin{proof}
	Let   $$R_P'=\{\br\in R_P: r_i-r_{i-1}\ge 2b^2 n \log T \mbox{ for } 1\le  i\le k    \mbox{ and  } r_k \le T- b n\log T  \},$$
	where $r_0=0$.
	It can be checked directly from the  definition that for every $\br \in R_P'$, 
	\begin{align}\label{eq;good}
	I_P(\br )=\{\bt \in \R^n: t_{2i-1}=r_i\mbox{ and } r_i\le t_{2i}\le r_i+bn\log T   \}. 
	\end{align}
	By Fubini's theorem and (\ref{eq;good})
	\begin{equation}
	\label{eq;kesou2}
	\begin{aligned}
	\int_{R_P'}\int_{I_P(\br )} \prod_{i=1}^k \mu(\phi_{\bt_i-r_{i, \bt}})  \dd \bt_{P} \dd \br&=\int_{R_P'}\Big(\int_0^{bn \log T} \mu(\phi_s\phi  )\dd s\Big)^{k
	} \dd \br . 
	\end{aligned}
	\end{equation}
Similar to (\ref{eq;main1}) and using Remark \ref{rem;simple} we have 
\begin{align}
\label{eq;easy}	
\left|\int_{R_P'}\int_{I_P(\br )} \int_X\prod_{i=1}^k \phi_{\bt_i}\dd \nu\dd \bt-
\int_{R_P'}\int_{I_P(\br )} \prod_{i=1}^n \mu(\phi_{\bt_i-r_{i, \bt}})  \dd \bt_{P} \dd \br\right|\le 1.
\end{align}
	
On the other hand, 	similar to (\ref{eq;RP2}),  we have 
	\begin{align}\label{eq;kesou}
	\int_{R_P\setminus R_P'}1 \dd \br \le 2b^2 n^2T^{k-1} \log T. 
	\end{align}
By Lemma \ref{lem;size},  (\ref{eq;kesou}) and the assumption $\|\phi\|_{\sup}\le 1$
	 \begin{align}\label{eq;coffee}
\left |b_{n, T, P}-\int_{R_P'}\int_{I_P(\br )} \int_X\prod_{i=1}^k \phi_{\bt_i}\dd \nu\dd \bt\right |\le \int_{R_P\setminus R_P'}\int_{I_P(\br)}1 \dd \bt\le 2 n^{n+2} b^{n^2+2} T^{k-1}   \log^{n+1} T.
	\end{align}
By Lemma \ref{lem;volume} and  (\ref{eq;kesou}) 
\begin{equation}\label{eq;music}
\frac{T^k}{k!}-\int_{R_P'}1 \dd \br=
\frac{T^k}{k!}-\int _{R_P}1\dd \br+\int_{R_P\setminus R_P'}1 \dd \br
\le 3b^n n^3T^{k-1} \log T.
\end{equation}
To sum up, by (\ref{eq;coffee})
\begin{align*}
\lim_{T\to \infty}\frac{b_{n, T, P} }{\frac{T^k}{k!}}&= \lim_{T\to \infty}
k! T^{-k}\int_{R_P'}\int_{I_P(\br )} \int_X\prod_{i=1}^k \phi_{\bt_i}\dd \nu\dd \bt\\
(\mbox{by \ref{eq;easy}})\qquad &= \lim_{T\to \infty}k! T^{-k}\int_{R_P'}\int_{I_P(\br )} \prod_{i=1}^n \mu(\phi_{\bt_i-r_{i, \bt}})  \dd \bt_{P} \dd \br\\
(\mbox{by \ref{eq;kesou2}})\qquad&=\lim_{T\to \infty}k! T^{-k}\int_{R_P'}
\Big(\int_0^{bn \log T} \mu(\phi_s\phi  )\dd s\Big)^{k
} \dd \br\\
(\mbox{by \ref{eq;music}})\qquad&=\lim_{T\to \infty}\Big(\int_0^{bn \log T} \mu(\phi_s\phi  )\dd s\Big)^{k
}\\
(\mbox{by Lemma \ref{lem;variance}(\rmnum{1})})\qquad&=\left(\frac{\sigma}{2}\right)^k.
\end{align*}
\end{proof}

\begin{lem}
	\label{lem;odd} If $n\ge 3$ is odd,   and $P=P_1$ is given by  (\ref{eq;odd}), then 
	\begin{align}\label{eq;oddp}
	\lim_{T\to \infty}\frac{b_{n, T, P} }{T^{\frac{n}{2}}}	=0.
	\end{align}
\end{lem}

\begin{proof}
Recall that 	 $k=|P|=\frac{n}{2}+\frac{1}{2}$ and $\|\phi\|\le 1$. 
By (\ref{eq;main1})	\begin{align*}
\left| b_{n, T, P}    \right|&\le 1+\int_{R_P}    \left |\int_{X}\phi_{r_1}\dd\nu \right|
\left(\int_{I_P(\br)}  \prod_{i=2}^k  \int_X |\phi_{\bt_i-r_{i, \bt}}|\dd \mu\dd \bt_P   \right)   \dd \br\\
(\mbox{by Lemma \ref{lem;size}})\qquad&\le 1+ b^{n^2}n^n (\log^n T)\int_{R_P}     \left |\int_{X}\phi_{r_1}\dd\nu \right|
\dd \br   \\
&\le 1+ b^{n^2}n^n (\log^n T)T^{k-1}\int_0^\infty  \left |\int_{X}\phi_{r_1}\dd\nu
\right|\dd r_1\\
(\mbox{by (\ref{eq;many})})\qquad&\le 1+ b^{n^2}n^n (\log^n T)T^{k-1}M  \int_0^\infty e^{-\delta t}\dd r_1.  
\end{align*}
So (\ref{eq;oddp}) follows from the above estimate and the observation $k-1=\frac{n-1}{2}<\frac{n}{2}$.	
\end{proof}

\begin{proof}
	[Proof of Theorem \ref{thm;central h}]
	
By Theorem \ref{thm;sencond} it suffices to prove (\ref{eq;moment}). As noted after Theorem 
\ref{thm;sencond} that  (\ref{eq;moment}) holds for $n=0, 1$ and $2$. 
For an odd integer  $n\ge 3$, one has 
\begin{align*}
\lim_{T\to \infty}\frac{b_{n, T}}{T^{\frac{n}{2}}} &=\Big(\sum_{|P|<\frac{n}{2}} +
\sum_{|P|>\frac{n+1}{2}}+\sum_{|P|= \frac{n+1}{2}}\Big)\lim_{T\to \infty}\frac{b_{n, T, P}}{T^{\frac{n}{2}}}.
\end{align*}
Then it follows from Corollary \ref{cor;small}, Lemma \ref{lem;separate} and Lemma \ref{lem;odd} that (\ref{eq;moment}) holds for odd $n\ge 3$. 
For an even integer $n \ge 3$, one has 
\begin{align*}
\lim_{T\to \infty}\frac{b_{n, T}}{T^{\frac{n}{2}}} &=\Big(\sum_{|P|<\frac{n}{2}} +
+\sum_{|P|>\frac{n}{2}}+\sum_{|P|= \frac{n}{2}}\Big)\lim_{T\to \infty}\frac{b_{n, T, P}}{T^{\frac{n}{2}}}.
\end{align*}
So it follows from Corollary \ref{cor;small}, Lemma \ref{lem;separate} and Lemma \ref{lem;equal} that (\ref{eq;moment}) holds for even $n\ge 3$.	
\end{proof}

\section{Regularity along stable and unstable leaves}\label{sec;leaf}

Let  $\Gamma $ be  a lattice of a connected semisimple Lie group $G$ with finite center and
$a\in G$ be $\Ad$-diagonalizable. 
  We assume in this section    that the action of $a$ on $X=G/\Gamma$ is ergodic with respect to 
  the probability Haar measure $\mu$. In view of the Mautner's phenomenon \cite[Thm.~III.1.4]{bekka}
  and the Ratner's measure classification theorem \cite[Thm.~3]{ratner1},  
  the ergodicity assumption  is  equivalent to $G_a'\Gamma$ is dense is
  $G$.  
  
 A function $\psi: Y\to \R$ on a metric space $(Y, \mathrm{dist})$ is said to be $\theta$-H\"older $(0<\theta\le 1)$  if 
  \begin{align*}
  \| \psi\|_{ \theta}':= \sup_{
	x, y\in Y,\;0<\mathrm{dist} (x, y)<1
  } \frac{|\psi(x)-\psi(y)|}{
  	\mathrm{dist} (x, y)^\theta}<\infty. 
  \end{align*}
  Here the upper bound of the distances between $x$ and $ y$  is only needed  in   the case where $Y$ is  noncompact.
Recall that we have fixed a right invariant metric on 
$G$ which induces a metric  on $X$. 
Any closed subgroup of $G$ is considered as a metric space with the metric inherited from that of   $G$.

   Let $H$ be a   closed subgroup of $G$ with Lie algebra $\mathfrak h\subset \mathfrak g$. Recall from 
   \S \ref{sec;review} that each $\alpha\in \mathfrak h^k$ defines a differential operator $\partial ^\alpha$ on
 $X$. 
   A function  $\varphi
  : X\to \R$ is said to be uniformly  smooth along $H$ orbits of a subset $X'\subset X$ if for any bounded open subset $U$ of $H$, any  $k\ge 0$ and any basis $\mathfrak b _{\mathfrak h}$ of $\mathfrak h$, there exists $M\ge 1$ such that
  \begin{align}
  \label{eq;consist}
  |\partial ^\alpha \varphi (gx)|\le M\quad \mbox {for all  }x\in X',g\in U \mbox{ and }
  \alpha\in \mathfrak b_{\mathfrak h}^k. 
  \end{align}
  It is not hard to see that to show (\ref{eq;consist}) holds for any $\mathfrak b _{\mathfrak h}$ it suffices to 
  prove it for a fixed $\mathfrak b _{\mathfrak h}$.


A useful tool to find  a continuous function which is equal to a given measurable function $\varphi:X\to \R$  is
to assign the density value at every point of measurable continuity. 
Recall that a point $x\in X$ is 
a point of measurable continuity of $\varphi$ if there is $s\in \R$ such that $x$ is a Lebesgue density point of $\varphi^{-1}(U)$ for every neighborhood $U$ of $s$ in $\R$. 
The value $s$ is called the  density value of $\varphi$ at $x$. 
Let $\mathrm{MC}(\varphi)$ be the set of measurable continuity points of $\varphi$, and let 
 $\til \varphi:\mathrm{MC}(\varphi)\to \R $
be the map which sends $x$ to the density value of $\varphi$ at $x$.

\begin{thm}
	\label{thm;general}

	Suppose  $( X, \mu, a)$ is ergodic and $\psi : X\to \R$ is $\theta$-H\"older. 
Let    $\varphi: X \to \R$ be a measurable solution   to the cohomological equation 
	\begin{align}\label{eq;new}
	\psi(x)=\varphi(ax)-\varphi(x).
	\end{align} 
	Let  $\til \varphi: \mathrm{MC}(\varphi)  \to 
	\R $ be the map which sends $x$ to the density value of $\varphi$ at $x$.
Then the followings hold: 
	\begin{enumerate}[\rm (i)]
		\item 	$ \mathrm{MC}(\varphi)$ is a  $G_a'$-invariant full measure subset of  $ X$ and $\varphi=\til \varphi$ almost everywhere;
		\item   $\til \varphi(gx)-\til \varphi (x)$ is continuous  on  $ G_a'\times  \mathrm{MC}(\varphi)$;
		\item   The H\"older norms   $\|\til \varphi(gx) \|_{\theta } \ (x\in  \mathrm{MC}(\varphi))$ of functions on $G_a'$ are uniformly
		bounded;
		\item   If $\psi\in\widehat  C_c^{\infty}(X)$,  then $\til \varphi$ is uniformly smooth along $G_a^-$ and  $G_a^+$ orbits of $ \mathrm{MC}(\varphi)$. 
	\end{enumerate}

\end{thm}

For every $x\in X, u\in G^-_a$ and $s\in \R $ we define the stable  holonomy map by 
\begin{align}
\label{eq;wu}
\mathcal H^-_{x, u}(s)=s+ \sum_{n=0}^\infty\psi(a^nx)-\psi(a^nu x).
\end{align}
In the context of \cite{asv},  $\mathcal H^-_{x, u}$ is a map from the fiber
of $x$ to the fiber of $ux$ in the fiber  bundle $X\times \R$. We do not use the fiber bundle language but adopt the name holonomy map. 
Similarly, for $u\in G^+_a$, the unstable holonomy map is defined by 
\begin{align}\label{eq;wu1}
\mathcal H^+_{x, u}(s)=s+ \sum_{n=-\infty}^{-1}\psi(a^nux)-\psi(a^nx).
\end{align}
Since $\psi$ is $\theta$-H\"older, 
both of the  
 holonomy maps are well-defined and continuous for $( u,x, s)\in  G_a^\pm\times X\times \R$.

A measurable function $\varphi: X\to \R$ is said to be essentially
$\mathcal H ^-$ invariant  if
 there is a full $\mu$ measure  subset $X'\subset X$ such that for all $x\in X, u\in G^-_a$ with 
$x, ux\in X'$, one has
\begin{align}\label{eq;jiazida}
\mathcal H^-_{x, u}(\varphi( x) )=\varphi(ux).
\end{align}
 We define 
  essentially $\mathcal H^+$ invariant in a similar way.

\begin{lem}
	\label{lem;essential}
	If  $\varphi$ is a measurable solution to the cohomological equation 
(\ref{eq;new}), 
 then 
	$\varphi$ is essentially $\mathcal H^+$ and   $\mathcal H^-$
	invariant. 
\end{lem}

\begin{proof}
	By Lusin's theorem, there is a
 compact subset  $K$ of $X$ such that $\varphi$ is uniformly continuous on $K$ and $\mu(K)>0.9$.
	Since $(X, \mu, a)$ is ergodic,  there exists an $a$-invariant
	full measure subset $X_1$ such that the Birkhoff average of the  characteristic function of $K$ at any $x\in X_1$  converges to $\mu(K)$. 
	By assumption, there is an $a$-invariant full measure  subset $X_2$ of $X$ such that  (\ref{eq;new}) holds for all $x\in X_2$. We will show  that $\varphi$ is essentially $\mathcal H^-$
invariant by taking $X'=X_1\cap X_2$. 

Suppose $ x\in X' , u\in G_a^-$ and $ux\in X'$.
Since $X_2$ is $a$-invariant, we have  $a^n x, a^n ux\in X_2  $  for all $n\in \N\cup \{0 \}$. 
 Using (\ref{eq;new}) for all  $a^n x$ and $ a^n ux$,  we have 
\begin{align}\label{eq;left}
\sum_{n=0}^\infty \psi(a^n x )-\psi(a^nu x)=\varphi(ux)-\varphi(x)+ \lim_{n\to \infty }\varphi(a^n x)-\varphi(a^nux),
\end{align}
where the existence of infinite  sum on the left hand of (\ref{eq;left}) follows form  the 
H\"older  property of  $\psi$.  So
 $\lim_{n\to \infty }\varphi(a^n ux)-\varphi(a^nx)$ converges. Since $x, ux\in X_1$, there are infinitely many $n$
 with $a^nx$ and $a^nu x$ belong to $K$.  Since $\varphi$ is uniformly continuous 
 on $K$ and $$\mathrm{d}( a^n ux , a^n x )\le \mathrm{d}_G( a^n ua^{-n} , 1_G )\to 0 \quad \mbox{as } n\to \infty,$$
we know that $\lim_{n\to \infty }\varphi(a^n ux)-\varphi(a^nx)=0$. 
 Therefore (\ref{eq;jiazida}) holds. 
  This proves that $\varphi$ is essentially $\mathcal H^-$ invariant. The proof of the  essential  
 $\mathcal H^+ $ invariance is similar. 
\end{proof}

\begin{lem}\label{lem;saturate}
	If  $\varphi$ is a measurable solution to the cohomological equation 
(\ref{eq;new}), 
then the  set $\mathrm{MC}(\varphi)$ is $G'_a$-invariant and 
\begin{align}\label{eq;simpler}
\til \varphi(ux)=\mathcal H^*_{x, u}(\til \varphi(ux))\quad \mbox{ for all } x\in \mathrm{MC}(\varphi) \mbox{ and  } u \in G_a^*
\end{align}
where $*\in \{+, -\}$. 
\end{lem}

\begin{proof}
	We show that $\mathrm{MC}(\varphi)$  is $G^-_a$-invariant and (\ref{eq;simpler})
	holds for $*=-$. One can prove the  $G^+_a$ invariance of $\mathrm{MC}(\varphi)$ and (\ref{eq;simpler}) for $*=+$ similarly. Since $G'_a$ is generated by $G^-_a$ and $G^+_a$, the 
	$G'_a$ invariance of  $\mathrm{MC}(\varphi)$  is a direct corollary. 
	The argument here is essentially the same as that in \cite{asv} but much simpler.  We provide details here for the completeness.

 By Lemma \ref{lem;essential}, 	the function  $\varphi$ is essentially $\mathcal H^-$ invariant. So there exists a full measure subset 
	$X'$ such that (\ref{eq;jiazida}) holds for all $x\in X'$ and $ux \in X'$ where $u\in G_a^-$.  
	Let $x\in \mathrm{MC}(\varphi)$ and $u\in G^-_a$. Suppose $s$ is the density value of 
	$\varphi$ at $x$. We will show that $\mathcal H^-_{x, u}(s)$ is the density value of $\varphi$
	at $ux$.   Let $U$ be a neighborhood of $\mathcal H^-_{x, u}(s)$. In view of the continuity
	of $\mathcal H^-_{*, *}(*)$, there exist open neighborhoods $N$ and $V$ of $x$ and $s$, respectively,   such that 
	\begin{align}
	\label{eq;tiaowu}
	\mathcal H^-_{y, u }(r)\in U\quad \mbox{ for all } (y, r)\in N\times V. 	
	\end{align}

	As $u: X\to X$ is a diffeomorphism  preserving $\mu$,  both $X'$ and $u^{-1}X'$ are full measure subsets. 
	Since $x\in \mathrm{MC}(\varphi)$,  it is a Lebesgue density point of 
	\[
	N\cap \varphi^{-1}(V)\cap X'\cap (u^{-1} X').
	\]
	It follows that $ux$ is a Lebesgue density point of 
		\[
N'=	u( N\cap  \varphi^{-1}(V)\cap X')\cap  X'.
	\]
	We claim that $N'\subset \varphi^{-1}(U) $. The claim will imply that  $ux$ is a point of measurable continuity  of $\varphi$ with density value $\mathcal H^-_{x, u}(s)$.

	To prove the claim let $uy$ be an arbitrary point of $N'$. So $uy\in X'$ and 
	$y\in N\cap  \varphi^{-1}(V)\cap X'$. Since $y, uy\in X'$, one has 
	$\mathcal H^-_{y, u}(\varphi(y))=\varphi( uy)$. 
	On the other hand, since $y\in N$ and $\varphi(y)\in V$, it follows from (\ref{eq;tiaowu})
	that $	\mathcal H^-_{y, u }(\varphi(y))\in U$. So $\varphi(uy)\in U$. This completes the proof of the claim and hence the lemma. 
\end{proof}

\begin{proof}[Proof of Theorem \ref{thm;general}]
	
	(\rmnum{1})
	By \cite[Lem. 7.10]{asv}, the set  $\mathrm{MC}(\varphi)$ has full  measure with respect to $\mu$ and 
	$\varphi=\til \varphi$ almost everywhere. 
	It follows from Lemma \ref{lem;saturate} that $\mathrm{MC}(\varphi)$ is $G_a'$-invariant. 

	(\rmnum{2})
	To simplify the notation, we set $X'=\mathrm{MC}(\varphi)$. In view of (\ref{eq;jiazida}) and (\ref{eq;simpler}), the map on $ G^-_a\times X'$ which sends $(u, x)$ to
	\begin{align}\label{eq;huawei}
	\til \varphi(ux)-\til \varphi(x)=\mathcal H^-_{x, u} (\til \varphi(x))-\til \varphi(x)=\sum_{n=0}^\infty \psi(a^n x)-\psi (a^n ux)	
	\end{align}
	is continuous on $ G^-_a\times X'$.
	Similarly $	\til \varphi(ux)-\til \varphi(x)$ is continuous for 
$(u, x)\in G^+_a \times X'$. 

In general, given $h\in G'_a$, there exists $m\in \N$ and  $\kappa_1, \ldots,\kappa_m\in \{ +, - \} $ such that the map 
\begin{align}\label{eq;surjective}
p: G^{\kappa_1 }_a \times \cdots \times G^{\kappa_m}_a\to G_a' \quad \mbox{given by  } p(g_1, \ldots, g_m)=g_m\cdots g_1
\end{align}
contains an open neighborhood of $h$, see \cite[\S 2.2]{bp}. Moreover, we can write  $h=h_m \cdots h_1$ for $h_i\in G^{\kappa_i}_a$ such that $p$ is an open map in a neighborhood of $(h_1, \ldots, h_m)$. Let 
\[
q: G^{\kappa_1 }_a \times \cdots \times G^{\kappa_m}_a\times X'\to \R
\]  
be the map  given by $q(g_1, \ldots, g_m, x)=\til \varphi(g_m\cdots g_1 x)-\til \varphi(x) $.  For  $g_0=1_G$ one has 
\begin{align}\label{eq;hx}
q(g_1, \ldots, g_m, x )=\sum_{n=1}^m \til \varphi(g_n\cdots g_1g_0 x  )-\til\varphi(g_{n-1}\cdots g_1g_0 x  ). 
\end{align}
So $q$ is continuous on $ G^{\kappa_1 }_a \times \cdots \times G^{\kappa_m}_a\times X'$. 
Since $p$ is an open map in a neighborhood of $(h_1, \ldots, h_m)$, the map 
$\til \varphi(gx)-\til \varphi(x)$ is continuous at  $(h, x)\in G_a'\times X'$. 
Therefore, $\til \varphi(gx)-\til \varphi(x)$  is a continuous function on $G'_a\times X'$.

(\rmnum{3})
Next we prove the uniform boundedness of the H\"older norms. Recall that  the metric on $G_a'$ is induced from the
right invariant  metric on   $G$.  So it suffices to prove that there is an  open neighborhood $U$ of the identity in  $G_a'$ and $M\ge 1$ such that for any $x\in X'$ and $g\in U$ 
\begin{align}\label{eq;holder}
|\til\varphi (gx)-\til \varphi(x)|\le M \mathrm{d}_{G}(g , 1_G)^\theta.
\end{align}

There is a map $p$ as in (\ref{eq;surjective}) such that the $p(1_G, \dots, 1_G)=1_G$ and the differential of $p$ at $(1_G, \dots, 1_G)$ is surjective. Therefore, there is a submanifold $V$ passing through $(1_G, \ldots, 1_G)$ such that $p|_V$ is a diffeomorphism onto its image.  In particular, $p|_V$ is a  bi-Lipschitz map onto $U=p(V)$.
So  
for all   $(g_1, \ldots, g_m)\in V$ and $g=p(g_1, \ldots, g_m)$ we have 
\begin{align}\label{eq;india}
\max_{i=1}^m \mathrm{d}_G (g_i, 1_G)\lesssim \mathrm{d}_G (g, 1_G) . 
\end{align}
So for $g_i\in G_a^-$ and $y\in X'$, by (\ref{eq;huawei}) and (\ref{eq;india}), we have
\begin{align}\label{eq;ant}
|\til \varphi(g_iy)-\til \varphi(y)|\lesssim \|\psi\|_\theta\mathrm{d}_G (g, 1_G)^\theta. 
\end{align}
The same estimate holds for $g_i\in G_a^+$.  
So for any $x\in  X'$ and $g\in U$, by (\ref{eq;hx}) and (\ref{eq;ant}), 
we have 
\begin{align*}
|\widetilde \varphi(gx)-\widetilde \varphi(x) |\le m \max_{n=1}^m  |\til \varphi(g_n\cdots g_1g_0 x  )-\til\varphi(g_{n-1}\cdots g_1g_0 x  )|
 \lesssim  m \|\psi \|_\theta \mathrm{d}_G (g, 1_G) ^\theta,
\end{align*}
from which (\ref{eq;holder}) follows.

(\rmnum{4})
Let $\mathfrak b_-$ be a basis of  the  Lie algebra of $G_a^-$ consisting of eigenvectors of $\Ad(a)$.
We assume without loss of generality that all the eigenvalues of $\Ad(a)$ are positive. 
 Suppose $\alpha=(v_1, \ldots v_r)\in \mathfrak b_-^r$ and $t_\alpha=\sum_{i=1}^rt_i$ where 
 $t_i>0$ satisfies  $\Ad(a)v_i 
 =e^{-t_i }v_i$. 
 We use $\partial ^\alpha_u$ to denote the differentiable operator $\partial ^\alpha$ on $G_a^-$ with respect to the 
 variable $u\in G_a^-$
Then 
 \begin{align}\label{eq;bird}
 \partial^\alpha_u \psi (a^nux)= \partial ^{\Ad (a^n)\alpha}\psi(a^n ux) =e^{-t_\alpha n} \partial ^\alpha\psi(a^n ux),
 \end{align}
whose sum over $n\ge 0$ converges uniformly for $u$ in a fixed compact subset and $x\in X$. 
So by (\ref{eq;huawei}) and (\ref{eq;bird}), for any $x\in \mathrm{MC}(\varphi)$ and $u\in G_a^-$, we have 
\begin{align*}
\partial^\alpha \til\varphi(ux)=-\sum_{n=0}^\infty e^{-t_\alpha n} \partial^\alpha \psi (a^nux). 
\end{align*}
Therefore   $\til \varphi$ is uniformly smooth  along $G^-_a$ orbits of $\mathrm{MC}(\varphi)$. By similar arguments one can prove that $\til \varphi$ is uniformly smooth  along $G^+_a$ orbits of $\mathrm{MC}(\varphi)$.

\end{proof}

\section{Liv\v sic type theorem}\label{sec;continuous}

Let the notation and assumptions be as in Theorem \ref{thm;continuous}. 
In particular, the action of $G_a'$
on $X=G/\Gamma$ has a spectral gap. This implies that the action of $G_a'$ on $(X, \mu)$ where 
$\mu$ is the probability Haar measure is mixing. As a consequence,  the dynamical system $(X, \mu, a)  $
is ergodic,  hence Theorem \ref{thm;general} holds. 

\begin{lem}
	\label{lem;solution}
	Suppose $\varphi$ is a  measurable solution to the cohomological equation $\psi(x)=\varphi(ax)-\varphi(x)$ where 
 $\psi \in \widehat C_c^\infty(X)$.  
Then $\varphi\in L^2_\mu$ and $ \mu(\psi)=0$. 
\end{lem}

\begin{proof}
  
    	Since the dynamical system $( X, \mu, a)$ is ergodic,  the measurable  solution $\varphi$ to the cohomological equation
	$ \psi(x)=\varphi(ax)-\varphi(x)$ is unique up to constants. Therefore, it suffices to prove that it has a 
 solution in 	$L^2_\mu$, i.e., $\psi$ is cohomologous to $0$ in $L_\mu^2$.  Since the action of $G_a'$ on $X$ has a spectral gap,  the dynamical system $(X,  \mu,
	 	a  )$ is exponential mixing of all orders by Lemma \ref{lem;orders}. 
Therefore, by \cite[Thm.~1.1]{bg} 
	the random variables 
	\[
	\frac{1}{n} \sum_{k=0}^{n-1} (\psi(a^kx)-\mu(\psi))=\frac{1}{\sqrt n} (\varphi(a^{n-1}x)-\varphi(x))-\sqrt n \mu(\psi)  \quad \mbox{given by }(X, \mu)
	\]
	converges as $n\to \infty$ to the  normal distribution with mean zero and  variance 
	\begin{align}\label{eq;sigma}
	\sigma =\int _X \psi(x)^2 \dd\mu(x)+2\sum_{n=1}^\infty \int_X \psi(a^nx)\psi(x)\dd\mu(x).
	\end{align}
  We claim that the random variables $\frac{1}{\sqrt n} (\varphi(a^{n-1}x)-\varphi(x))$ 
	converges to zero in distribution  as $n\to \infty$. To see this, given $\varepsilon>0$, there is $M\ge 1$ such that 
	$\mu (\{x\in X: \varphi(x)\le M \})>1-\varepsilon$. Since the measure $\mu$ is $a$-invariant,  for any $n>0$
	\[
	\mu\big (\{ x\in X: \varphi(a^{n-1}x)\le M \mbox{ and } \varphi(x)\le M  \}\big )> 1-2\varepsilon, 
	\]
	from which the claim follows. Therefore, the central limit theorem  has to be degenerate and $\mu(\psi)=0$.
It follows from \cite[Thm.~2.11.3]{pu} that $\psi$ is cohomologous to $0$ in $L^2_\mu$. 
\end{proof}

 Recall that an  element  $v\in \mathfrak g$
is identified with the a right invariant vector field on $G$ and it descents 
to a   vector field on $X$ with the same notation $v$. 

\begin{lem}
	\label{lem;distribution}
	Suppose $v\in \mathfrak g$ and $\psi, \phi\in \widehat C_c^\infty (X)$. 
	Then \begin{align}\label{eq;distribution}
	\int_X \partial ^v \psi \phi \dd\mu=\int_X  \psi \partial ^v\phi \dd \mu.
	\end{align}
\end{lem}

\begin{proof}
For any nonzero real number $t$ and $x\in X$ we let \[
f_t(x)=\frac{\psi(\exp (tv)x )\phi(\exp (tv)x )-\phi(x)\psi (x)}{t}.
\]
The functions $f_t$ are uniformly bounded by the mean value theorem. So the dominated convergence theorem implies
\[
\lim_{t\to 0}\int_X f_t \dd\mu=\int_X \lim_{t\to 0} f_t \dd\mu=\int_X 
\partial ^v \psi \phi- \psi \partial ^v\phi \dd \mu.
\]
On the other hand, since $G$ preserves $\mu$, we have $\int_X f_t \dd\mu=0$
for all $t$. 
Therefore (\ref{eq;distribution}) holds. 
\end{proof}

Lemma \ref{lem;distribution} allows us to define the distribution derivative of $\varphi\in L_\mu^2$. 
For each $\alpha\in \mathfrak g^k$, we define $\partial ^\alpha \varphi$ as a linear functional on $\widehat C^\infty_c(X)$ such that 
\[
\langle  \partial ^\alpha \varphi,  \phi \rangle =(-1)^{|\alpha|}
\langle \varphi, \partial ^\alpha   \phi \rangle  \quad \mbox{where}\quad |\alpha|=k. 
\]
In view of Lemma \ref{lem;distribution}, 
this definition is consistent with the usual definition in the sense that the distribution 
derivative of a function $\psi \in \widehat C_c^\infty(X)$ is represented by the function
$\partial ^\alpha \psi$.
 This property implies that our definition  of distribution derivative is independent of the choice of coordinates and  coincides with the definition using local charts. 
Let $\mathfrak g^-$ and $\mathfrak g^+$ be the Lie algebras of $G_a^-$ and $G_a^+$, respectively.
Let   $\mathfrak l $ be the Lie algebra of 
  $L=\{g\in G: ga=ag \}$. For a square integrable function $\xi $ on $L$ and $\alpha\in \mathfrak l^k$, one can define the distribution derivative $\partial ^\alpha \xi$ in a similar way. Moreover, this definition  coincides with the definition using local charts.

\begin{lem}
	\label{lem;partial}
	Suppose $\psi \in \widehat C_c^\infty (X)$ and 
 $\varphi$ is an $L^2_{\mu,0}$ solution to the cohomological equation $\psi(x)=\varphi(ax)-\varphi(x)$. 
 Then for all $\alpha\in \mathfrak l^k$ and $\phi\in \widehat C_c^\infty (X)$, one has
	\[
\langle	\partial ^\alpha \varphi, \phi\rangle=-\sum_{n=0}^\infty \int_X \partial^\alpha \psi (a^n x)\phi(x)\dd\mu(x). 
	\]
\end{lem}

\begin{proof}
	According to the definition and the assumption $\psi(x)=\varphi(ax)-\varphi(x)$, 
	\begin{equation}\label{eq;luan}
	\begin{aligned}
	\langle	\partial ^\alpha \varphi, \phi\rangle&=(-1)^{|\alpha|}
	\int_X \varphi \partial ^\alpha \phi \dd\mu \\
	&=(-1)^{|\alpha|}\int_X \big (\varphi (a^{k}x)-\sum_{n=0}^{k-1}  \psi(a^n x)\big)\partial^\alpha\phi(x)
	\dd\mu(x)\\
	&=(-1)^{|\alpha|}\left(\int_X \varphi (a^kx )\partial^\alpha\phi(x) \dd \mu-
	\sum_{n=0}^{k-1}\int_X \psi (a^nx )\partial^\alpha\phi(x) \dd \mu
	\right).
	\end{aligned}
	\end{equation}
	As the action of $a$ on $(X, \mu)$ is  mixing and $\mu(\varphi)=0$, we have 
$$\lim_{k\to \infty} \int_X \varphi (a^kx )\partial^\alpha\phi(x) \dd \mu=0.$$
On the other hand, by  Lemma \ref{lem;mixing}  the mixing is  exponential for functions in $\widehat C_c ^\infty (X)$.
So by (\ref{eq;luan}) and  Lemma \ref{lem;distribution},
	\begin{align*}
\langle	\partial ^\alpha \varphi, \phi\rangle&=(-1)^{|\alpha|+1}
	\sum_{n=0}^\infty\int_X \psi (a^nx )\partial^\alpha\phi(x) \dd \mu\\
&	=-\sum_{n=0}^\infty \int_X \partial^\alpha(\psi \circ a^n )(x)\phi(x)\dd\mu(x)\\
& = -\sum_{n=0}^\infty \int_X \partial^\alpha\psi ( a^n x)\phi(x)\dd\mu(x).
\end{align*}
\end{proof}

A vector field $v$ on $X$ is said to be tangent to $L$ orbits if
its value $v_x$ at  each point  $x\in X$  is tangent to the submanifold $Lx$ .
The space of all such  vector fields is denoted by $\mathcal F_L(X)$. 
We fix a basis  $\mathfrak b$ of $\mathfrak g$ consisting of eigenvectors of $\Ad (a)$ and let  $\mathfrak b _{\mathfrak l}=\mathfrak b\cap \mathfrak l$. 

\begin{lem}
	\label{lem;commute}
	Suppose $\beta\in \mathcal F_L(X)^k$. Then there exists a 
	family of smooth functions $f_\alpha\ (\alpha\in \mathfrak b^r_{\mathfrak l}, r\le k)$ 
	such 
	that 
	\[
	\partial ^\beta =\sum_{\alpha\in \mathfrak b_{\mathfrak l}^r, r\le k} f_\alpha \partial ^\alpha.
	\]  
\end{lem}
\begin{proof}
	Since the values of $\mathfrak b_{\mathfrak l}$ at any point form a basis of the tangent space of the point,  the conclusion is clear
for  $k=1$. 
We prove the general case by induction.  Suppose the conclusion holds  for $|\beta|\le k$ 
and $\partial ^\beta = \partial^{v_1}\cdots \partial ^{v_{k+1}} $ where $v_i\in \mathcal F_L(X)$.  
By the  induction hypothesis 
\[
\partial ^{v_1}=\sum_{v\in \mathfrak b_{\mathfrak l}} f_v \partial ^v\quad \mbox{
and } \quad 
\partial^{v_2}\cdots\partial ^{v_{k+1}}=\sum_{\alpha\in \mathfrak b_{\mathfrak l}^r, r\le k} f_\alpha \partial ^\alpha, 
\] 
where $f_v$ and $f_\alpha$ are smooth functions on $X$.
Let $\phi\in C_c^\infty (X)$ be a test function.
Then \begin{align*}
\partial ^\beta \phi=\sum_{v\in \mathfrak b_{\mathfrak l}} f_v \partial ^v \big(\sum_{
	|\alpha |\le k } f_\alpha \partial ^\alpha \phi \big)
=\sum _{|\alpha|\le k} \big (\sum _{v\in \mathfrak b_{\mathfrak l}}f_v\partial ^v f_\alpha \big )\partial ^\alpha\phi+\sum_{v\in \mathfrak b_{\mathfrak l}} f_vf_\alpha \partial ^{(v,\alpha)} \phi.
\end{align*}
So the conclusion holds for $|\beta|=k+1$. 
\end{proof}

\begin{proof}
	[Proof of Theorm \ref{thm;continuous}]
We have proved  $\varphi\in L_\mu^2$ in Lemma \ref{lem;solution}. Next we show that $\varphi$ is  equal to 
 a smooth function
 almost surely. By possibly replacing $\varphi$ by $\varphi-\mu(\varphi)$, we assume that $\mu(\varphi)=0$.    In view of  Theorem \ref{thm;general}, we assume that $\varphi=\til \varphi$ is defined on 
the  $G_a'$-invariant full measure subset $X'=\mathrm{MC}(\varphi)$ so that the conclusions of Theorem \ref{thm;general} hold.
We will show that $\mathrm{MC}(\varphi)=X$ and   $\varphi$ is smooth.

The question is local, so we 
 fix  $x\in X$. 
  We choose   open neighborhoods  $U$,  $U^-$ and $U^+$ of the  identity in $L$,  $G_a^-$ and $G_a^+$,  respectively, such that $U\times U^-\times U^+$ is diffeomorphic to its image  via the map $(g, h_0, h_1)\to gh_0h_1x$.  
Let 
$
\xi : U  \to X  $  {be defined by } $ g\in U\to \varphi (gx). 
$
 We fix Haar measures on $G, G_a^-, G_a^+$ and $L$ so that a fundamental domain of $\Gamma$ has measure $1$ and 
\begin{align}\label{eq;haar}
\dd (gh_0h_1)=\dd g \dd h_0 \dd h_1\quad \mbox{ where }
\quad g\in L, h_0\in G_a^- \mbox{ and }h_1\in 
G_a^+. 
\end{align}
Since   $\varphi\in L_\mu^2$,  the function $\xi $  is square integrable.

  Suppose $g\in L$ and $gx\in X'$. Note that  $L$ normalizes $G_a^+$ and $G_a^-$. So by (\ref{eq;simpler}), 
  (\ref{eq;wu}) and (\ref{eq;wu1})
  \begin{equation}
  	\label{eq;conference}
  	\begin{aligned}
  		&\varphi(gh_0h_1x)=\varphi(gh_0g^{-1} \cdot gh_1 x)\\
  		=&\varphi(gh_1x)+\sum_{n=0}^{\infty}\psi (a^n gh_1 x)-\psi(a^n gh_0h_1 x)\\
  		=&\varphi(gh_1g^{-1} \cdot gx)+\sum_{n=0}^{\infty}\psi (a^n gh_1 x)-\psi(a^n gh_0h_1 x)\\
  		=& \varphi(gx)+\lambda_\psi( g, h_0, h_1),
  	\end{aligned}
  \end{equation}
  where 
  \begin{equation}
  \label{eq;conference0}
  \begin{aligned}
  	\lambda_\psi(g,  h_0, h_1)=\sum_{n=-\infty}^{-1}[\psi(a^n gh_1 x)-\psi(a^ngx)]+\sum_{n=0}^{\infty}[\psi (a^n gh_1 x)-\psi(a^n gh_0h_1 x)]\\
  	=\sum_{n=-\infty}^{-1}[\psi(ga^n h_1 x)-\psi(ga^nx)]+\sum_{n=0}^{\infty}[\psi (ga^n h_1 x)-\psi(ga^n h_0h_1 x)].
  \end{aligned}
    \end{equation}
  For $\alpha\in \mathfrak l^k$, let $\partial^\alpha_g$ be the differential operator $\partial ^\alpha$ with respect to the variable   $g\in L$. Then 
  \begin{equation}
  \label{eq;conference1}
  \begin{aligned}
  \partial^\alpha_g [\psi(ga^n h_1 x)-\psi(ga^nx)]&=\partial^\alpha\psi(ga^n h_1 x)-\partial^\alpha\psi(ga^nx)\\
  &=\partial^\alpha\psi(a^n gh_1g^{-1} \cdot gx)-\partial^\alpha\psi(a^n \cdot gx),
  \end{aligned}
    \end{equation}
  whose sum over negative integers $n$ converges uniformly with respect to $h_1\in U^+$ and $g\in U$. 
  Similar statement holds for $\partial^\alpha_g[\psi (ga^n h_1 x)-\psi(ga^n h_0h_1 x)]$.
Therefore, (\ref{eq;conference0}) converges uniformly on $U\times U^-\times U^+$ and 
  \begin{align}\label{eq;conference2}
  \partial^\alpha_g \lambda_\psi(g, h_0, h_1 )=\lambda_{\partial^\alpha\psi}(g, h_0, h_1).
  \end{align}

The uniform convergence of   (\ref{eq;conference0}) implies that $\lambda(g, h_0, h_1)$ is continuous on 
$U\times U^-\times U^+$. 
In view of  (\ref{eq;conference}), if $\xi$ is equal to a  continuous function almost everywhere,  then $\varphi$ is equal to a continuous function almost everywhere. 
This implies that $MC(\varphi)\cap UU^-U^+ x=UU^-U^+ x$ and 
$\varphi $ is continuous on $UU^-U^+x$.
It follows from   (\ref{eq;conference}), (\ref{eq;conference0}),  (\ref{eq;conference1}) and (\ref{eq;conference2}) that if 
$\xi$ is smooth on $U$, then for any $\alpha\in \mathfrak l^k$ we have  $\partial ^\alpha\varphi(y) $ exists for all
$y\in UU^-U^+ x$ and these values are uniformly bounded on  $UU^-U^+x$. 
On the other hand, 
   by Theorem \ref{thm;general}(\rmnum{4}), 
   the function $\varphi$ is uniformly smooth along $G^-_a$ and $G_a^+$ orbits of $\mathrm{MC}(\varphi)$. 
 Therefore,  if $\xi$ is equal to a smooth function almost everywhere, then $\varphi$ is defined everywhere 
 on $UU^-U^+ x$ and 
all the partial derivatives of $\varphi$ along foliations  are uniformly bounded on $UU^-U^+ x$. So a theorem of 
  Journ\'e \cite{journe} implies  that $\varphi$ is smooth on $UU^-U^+ x$.

 Therefore, it suffices to show that $\xi$ is equals to a smooth function almost everywhere.
 Let $m$ be the dimension of $L$.
 We assume that there is a coordinate map $b: U \to \Omega$ where $\Omega$ is an open ball in
 $\R^m$. Moreover, we assume there is a smooth function $\rho: U\to \R $ that  is  bounded from above and below by some positive constants  so that $\rho(g) \dd g $ is mapped by $b$ to the Lebesgue measure on 
 $\Omega$.   
 For a compactly supported smooth function $\zeta_0: U\to \R$, we consider the Fourier transform of $f(y)=\xi(b^{-1}y)\zeta_0 (b^{-1}y)$ defined by 
\[
\widehat f(z)= \int_{\Omega} \xi(b^{-1}y) \zeta_0(b^{-1}y) e^{-2\pi i y\cdot z}\dd y, 
\]
which is a continuous function on $\R^m$. 
We claim that for any positive integer $n$ and $z=(z_1, \ldots, z_m)\in \R^m$, we have
\begin{align}\label{eq;claim}
|z|^n_{\sup} \widehat f(z)\in L^1(\R^m)\quad \mbox{where} \quad  |z|_{\sup} =\max \{ |z_1|, \ldots, |z_m| \}
. 
\end{align}
Assume the claim, then  it follows from the Fourier inversion formula that $f$ equals to a smooth function almost everywhere, see for example Folland \cite[Thm.~8.22.d, 8.26]{folland}.
By choosing $\zeta_0$ with arbitrarily large support,  we have $\xi$
is equal to a smooth function almost everywhere.

Now we prove the claim.  We will use the usual multiple index notation for the partial derivatives on $\R^m$ (see \cite[\S 8.1]{folland}).   For example, if
$\gamma=(k_1, \cdots, k_m)\in \Z_{\ge 0}^m $ is a multiple index, then 
 $$
z^\gamma=z_1^{k_1}\cdots z_m^{k_m}, \quad \partial ^\gamma =\left(\frac{\partial}{\partial y_1}\right)^{k_1}\cdots \left(\frac{\partial}{\partial y_m}\right)^{k_m}\quad \mbox{and}\quad |\gamma|=k_1+\cdots k_m.
 $$
We have
\begin{equation}\label{eq;sign}
\begin{aligned}
(2\pi i  z)^\gamma \widehat f(z)&=(-1)^{|\gamma|}\int_{\Omega} \xi(b^{-1}y) \zeta_0(b^{-1}y)\partial ^\gamma e^{-2\pi i y\cdot z}\dd y\\
&=\int_{\Omega} \partial ^\gamma\big(\xi(b^{-1}y) \zeta_0(b^{-1}y)\big) e^{-2\pi i y\cdot z}\dd y,
\end{aligned}
\end{equation}
where $\partial ^\gamma$ refers to the distribution derivative.
By the product rule
\begin{equation}\label{eq;sign1}
\begin{aligned}
\partial ^\gamma\big(\xi(b^{-1}y)\zeta_0(b^{-1}y)\big)&=\sum_{\tau\le \gamma }
\partial ^\tau\xi(b^{-1}y)\partial ^{\gamma-\tau} \zeta_0(b^{-1}y), 
\end{aligned}
\end{equation}
where the sum is taken over all the multiple indices $\tau $ such that  $\gamma-\tau$ has nonnegative entries and each $\tau$ appears 
$\frac{\gamma!}{\tau!(\gamma-\tau)!}$ times. 
Since $\zeta_0$ is a
 compactly supported function on $U$, there are 
 differential operators 
$\til \tau \in \mathcal F_L^{|\tau|}$ and $\til {\gamma-\tau }\in \mathcal F_L^{|\gamma-\tau|} $ such that 
\begin{align}\label{eq;sign2}
\partial ^\tau\xi(b^{-1}y)\partial ^{\gamma-\tau} \zeta_0(b^{-1}y)=\partial ^{\til\tau}\xi(g)\partial ^{\til{\gamma- \tau}} \zeta_0(g)
\end{align}
where $b(g)=y$.

By (\ref{eq;sign}),  (\ref{eq;sign1}) and (\ref{eq;sign2})
\begin{align}\label{eq;dance}
(2\pi i  z)^\gamma\widehat f(z)&=\int_{U} \sum_{\tau\le \gamma}\partial ^{\til\tau}\xi(g)\partial ^{\til{\gamma- \tau}} \zeta_0(g) 
e^{-2\pi i b(g)\cdot z}\rho(g)\dd g.
\end{align}
By Lemma \ref{lem;commute}, for each $\til \tau$ one has 
\begin{align}\label{eq;dance1}
\partial ^{\til \tau}=\sum_{\alpha \in \mathfrak b_{\mathfrak l}^r , r\le |\tau|}
f_\alpha\partial ^\alpha,
\end{align}
where $f_\alpha$ are smooth functions.  
In view of (\ref{eq;dance}) and (\ref{eq;dance1}) 
\begin{align}\label{eq;april}
(2\pi i  z)^\gamma\widehat f(z)=\sum_{\alpha \in \mathfrak b_{\mathfrak l}^r,r\le |\gamma|}\int_{U} \partial ^{\alpha}\xi(g)\eta_{\alpha}(g)e^{-2\pi i b(g)\cdot z}\dd g,
\end{align}
where $\eta_{\alpha}(g)$ are  compactly supported smooth functions on $U$.

Now we   lift all the $\eta_\alpha$ to  compactly supported smooth functions $\til \eta _\alpha$ on $X$. 
Let $\zeta: U^{-}\times U^+\to [0,\infty)$ be a compactly supported smooth function such that 
$$\int_{U^+}\int_{U^-} \zeta (h_0,h_1)\dd h_0\dd h_1=1 .$$ 
Each function  $\eta \in C_c^\infty(U)$ is lifted to  $\til \eta \in C_c^\infty (X)$ as follows: 
\[
\eta(y)=\left \{ 
\begin{array}{ll}
\eta (g)\zeta (h_0, h_1) & \mbox{if } y=gh_0h_1x\mbox{ for }g\in U,h_0\in U^-, h_1\in U^+\\
0 & \mbox{otherwise}
\end{array}
\right.   .
\]
By (\ref{eq;haar}), (\ref{eq;conference}) and the  Fubini's theorem  we have 
\begin{equation}\label{eq;faint}
\begin{aligned}
\int_X \varphi(x) \til \eta(x)\dd \mu(x)&=\int_{U^+}\int_{U^-}\int_U\varphi(gh_0h_1 x) \til \eta(gh_0h_1 x) 
\dd g\dd h_0 \dd h_1\\
 & =\int_U\int_{U^+}\int_{U^-} \big(\varphi(gx)+\lambda_\psi(g,  h_0, h_1)\big)\eta(g )\zeta(h_0,h_1)\dd h_0\dd h_1\dd g\\
& =\int_U\lambda_\psi(g)  \eta(g) \dd g+\int_U \xi(g)\eta(g)\dd g,
\end{aligned}
\end{equation}
where 
\[
\lambda_\psi(g) =\int_{U^+}\int_{U^-}\lambda_\psi(g, h_0, h_1)\zeta(h_0, h_1)\dd h_0\dd h_1. 
\]

By (\ref{eq;conference2}) and the uniform convergence of (\ref{eq;conference0}), the function $\lambda_\psi(g)$ is  smooth on $U$ and 
for  any $\alpha\in \mathfrak b_{\mathfrak l}^r$ 
\begin{align}\label{eq;visa}
\partial^\alpha\lambda_{\psi}(g)=\lambda_{\partial^\alpha\psi} (g). 
\end{align}
 For any  $\eta \in C_c^\infty(U)$, by  (\ref{eq;faint}) and (\ref{eq;visa}),  we have 
\begin{equation}\label{eq;wumai}
\begin{aligned}
&\int_U (\partial ^\alpha \xi) \eta\dd g=(-1)^r\int_U \xi(\partial ^\alpha  \eta)\dd g\\ 
=& (-1)^r\int_X \varphi(x) \til {\partial^\alpha\eta}(x)\dd \mu(x)+(-1)^{r+1}\int_U\lambda_\psi(g)  (\partial^\alpha\eta)(g) \dd g\\
=&\int_X \partial ^\alpha\varphi(x) \til {\eta}(x)\dd \mu(x)-\int_U \partial ^\alpha \lambda _\psi(g) \eta(g)\dd (g)\\
=&  \int_X \partial ^\alpha\varphi(x) \til {\eta}(x)\dd \mu(x)-\int_U \lambda _{\partial ^\alpha \psi}(g) \eta(g)\dd (g).
\end{aligned}
\end{equation}

Write $\eta_{\alpha,z}(g)=\eta_{\alpha}(g)e^{-2\pi i b(g)\cdot z}$, then 
by (\ref{eq;wumai})
\begin{align}\label{eq;april1}
\int_{U}\partial ^{\alpha}\xi(g)\eta_{\alpha,z}(g)\dd g=\int_X \partial ^{\alpha}\varphi(x)\til \eta_{\alpha,z} (x)\dd\mu(x)-\int_U  \lambda _{\partial ^\alpha\psi}(g)\eta _{\alpha, z}(g)\dd g. 
\end{align}
By Lemma \ref{lem;partial}, 
\begin{equation}\label{eq;april2}
\begin{aligned}
\left|\int_X \partial ^{\alpha}\varphi(x)\til \eta_{\alpha,z} (x)\dd\mu(x)\right|&=
\left|
\sum_{n=0}^\infty \int_X \partial ^\alpha\psi (a^n x)\til  \eta_{\alpha,z} (x)\dd\mu(x)\right|\\
& \le  \sum_{n=0}^\infty\left |\int_X \partial ^\alpha\psi (a^n x)\til  \eta_{\alpha,z} (x)\dd\mu(x)\right|.
\end{aligned}
\end{equation}
By Lemma  \ref{lem;mixing},
\begin{align}\label{eq;march}
\left |\int_X \partial ^\alpha\psi (a^n x)\til  \eta_{\alpha,z} (x)\dd\mu(x) \right |
\lesssim_{\alpha}  e^{-n\delta' }\|\eta_{\alpha,z}\|_{\ell_0}\lesssim_\alpha e^{-n\delta' }|z|_{\sup}^{\ell_0}.
\end{align}
So (\ref{eq;april2}) and (\ref{eq;march}) imply 
\begin{align}
\label{eq;eye}
\left|\int_X \partial ^{\alpha}\varphi(x)\til \eta_{\alpha,z} (x)\dd\mu(x)\right|\lesssim_\alpha |z|_{\sup}^{\ell_0}. 
\end{align}
On the other hand, 
\begin{align}
\label{eq;tired}
\left |\int_U  \lambda _{\partial ^\alpha\psi}(g)\eta _{\alpha, z}(g)\dd g\right |\le \int_U  
\left | \lambda _{\partial ^\alpha\psi}(g)\eta _{\alpha}(g)\right |  \dd g \lesssim_{\alpha}1.
\end{align}
By   (\ref{eq;april1}), (\ref{eq;eye}) and (\ref{eq;tired}),  we have 
\begin{align}
\label{eq;night}
\left|\int_{U}\partial ^{\alpha}\xi(g)\eta_{\alpha,z}(g)\dd g\right|\lesssim_{\alpha} |z|_{\sup}^{\ell_0}+1.
\end{align}
Using (\ref{eq;night}) to estimate each term of the right hand side of (\ref{eq;april}), we get 
\begin{align}\label{eq;mid}
|(2\pi i z)^\gamma \widehat f(z)|\lesssim_{\gamma} |z|_{\sup}^{\ell_0}+1.
\end{align}
For any positive integer $k$, 
we sum the left hand side of (\ref{eq;mid}) for all $\gamma=2k \gamma_i$ where $\{\gamma_i:1\le i\le m\}$
is the standard basis of $\R^m$, then
\[
|z|^{2k}_{\sup} |\widehat f(z)| \lesssim_k |z|_{\sup}^{\ell_0}+1,
\]
from which the claim (\ref{eq;claim}) follows. 

\end{proof}

\begin{proof}
	[Proof of Theorem \ref{thm;checkable}] 
	Suppose  the variance $\sigma(F, \phi)$ is zero, then Theorem \ref{thm;variance} implies that there is a measurable solution $\varphi$ to the system of cohomological equations (\ref{eq;guoqing}). Note that $\psi (x)=\int_0^1 \phi (a_tx )\dd t$
	belongs to $\widehat  C_c^\infty(X)$ and $\varphi$ is a measurable solution to the cohomological equation $\psi(x)=\varphi(a_1x)-\varphi(x)$. By assumption, the projection of $F$ to each simple factor of $G$ is nontrivial, so the action of $G_{a_1}'=H$ on $X$ has a spectral gap. 
    Theorem \ref{thm;continuous} implies that there is a smooth function $\til \varphi$ on 
	$X$ such that $\til \varphi=\varphi$ almost everywhere. This means that there is a continuous  function $\varphi$ such that  (\ref{eq;guoqing}) holds for all $s> 0$ and $x\in X$.

	To prove $\phi$ is dynamically null with respect to  $(  X, F)$, it suffices to show that for any $F$-invariant and ergodic probability measure $\til \mu$ on $X$ the integral $\til \mu(\phi)=0$. 
	By  Birkhoff ergodic theorem and Poincar\'e recurrence theorem, there is $x\in X$ such that $x $ is in the closure of $\{ a_t x: t\ge T_0 \}$ for any  $T_0\ge 0$ and 
	\begin{align*}
		\til \mu(\phi)=\lim_{T\to \infty}\frac{1}{T}\int_0^T \phi(a_tx)\dd t=\lim_{T\to \infty} \frac{1}{T}(\varphi(a_Tx)-\varphi(x)).
	\end{align*}
	The above two properties together with the continuity of $\varphi$ imply  $\til \mu(\phi)=0$. 
\end{proof}

\end{document}